\documentclass[12pt,a4paper,reqno]{amsart}

\usepackage{amsmath,amsthm,amsfonts,amssymb,amscd,verbatim,delarray,url,epigraph06}
\usepackage[arrow,matrix]{xy}
\usepackage[croatian,english]{babel}

\usepackage{mathtools}

\numberwithin{equation}{section}

\DeclareMathOperator{\PSL}{PSL} \DeclareMathOperator{\SL}{SL}
 \DeclareMathOperator{\GL}{GL}
\DeclareMathOperator{\tr}{tr} 
 
\DeclareMathOperator{\id}{id} \DeclareMathOperator{\SU}{SU}

\DeclareMathOperator{\Cr}{Cr}
\DeclareMathOperator{\M}{M}

\DeclareMathOperator{\ch}{char}

\begin{document}

\theoremstyle{plain}
\newtheorem{theorem}{Theorem}[section]
\newtheorem{conj}[theorem]{Conjecture}
\newtheorem{corollary}[theorem]{Corollary}
\newtheorem{prop}[theorem]{Proposition}
\newtheorem{Lemma}[theorem]{Lemma}
\newtheorem{lemma}[theorem]{Lemma}

\theoremstyle{definition}
\newtheorem{definition}[theorem]{Definition}
\newtheorem{quest}[theorem]{Question}
\newtheorem{remark}[theorem]{Remark}
\newtheorem{example}[theorem]{Example}
\newtheorem*{answer}{Answer}
\newtheorem{notation}[theorem]{Notation}
\newtheorem*{prin}{Panda Principle}
\newtheorem{question}[theorem]{Question}
\newtheorem{obs}[theorem]{Observation}

\newcommand{\thmref}[1]{Theorem~\ref{#1}}
\newcommand{\secref}[1]{Section~\ref{#1}}
\newcommand{\subsecref}[1]{Subsection~\ref{#1}}
\newcommand{\lemref}[1]{Lemma~\ref{#1}}
\newcommand{\corolref}[1]{Corollary~\ref{#1}}
\newcommand{\exampref}[1]{Example~\ref{#1}}
\newcommand{\remarkref}[1]{Remark~\ref{#1}}
\newcommand{\defnref}[1]{Definition~\ref{#1}}
\newcommand{\propref}[1]{Proposition~\ref{#1}}

\newcommand{\BZ}{\mathbb{Z}}
\newcommand{\BC}{\mathbb{C}}
\newcommand{\BN}{\mathbb{N}}
\newcommand{\BP}{\mathbb{P}}
\newcommand{\BF}{\mathbb{F}}
\newcommand{\BA}{\mathbb{A}}
\newcommand{\BQ}{\mathbb{Q}}
\newcommand{\BM}{\mathbb{M}}
\newcommand{\BX}{\mathbb{X}}
\newcommand{\BR}{\mathbb{R}}

\newcommand{\ep}{\varepsilon}
\newcommand{\al}{\alpha}
\newcommand{\be}{\beta}
\newcommand{\ga}{\gamma}
\newcommand{\de}{\delta}
\newcommand{\la}{\lambda}
\newcommand{\om}{\omega}
\newcommand{\vp}{\varphi}
\newcommand{\st}{\sigma}

\newcommand{\vareps}{\varepsilon}

\newcommand{\G}{\Gamma}

\newcommand{\ov}{\overline}

\newcommand{\fc}{\frac}

\newcommand\cR{\mathcal R}
\newcommand\cP{\mathcal P}
\newcommand\cF{\mathcal F}
\newcommand\cA{\mathcal A}
\newcommand\cG{\mathcal G}
\newcommand{\Fg}{\mathfrak g}
\newcommand{\Fsl}{\mathfrak sl}

%\begin{document}

\newcommand{\lra}{\longrightarrow}
\newcommand{\idd}{\mathop{\rm{id}}\nolimits}
\newcommand{\R}{{\mathbb R}}
\newcommand{\C}{{\mathbb C}}
\newcommand{\Z}{{\mathbb Z}}
\newcommand{\N}{{\mathbb N}}
\newcommand{\Q}{{\mathbb Q}}

\newcommand{\e}{{\mathfrak{e}}}

\newcommand{\Om}{{\overline{\Omega}_{a_1a_2\cdots a_s}}}
\newcommand{\Oms}{{\Omega}_{a_1a_2\cdots a_s}}
\newcommand{\Fe}{{F}_{a_1a_2\cdots a_s}}
\newcommand{\Imm}{\operatorname{Im}}
\newcommand{\Omsr}{{\mathfrak{T}_r\Omega}_{a_1a_2\cdots a_s}}
\newcommand{\Omr}{{\mathfrak{T}_r\overline{\Omega}_{a_1a_2\cdots
a_s}}}

%\title{Maps}

%\author{}

\newcommand{\din}{{\rm d}_{\rm in}}
\newcommand{\E}{\mathcal{E}}
\newcommand{\Ma}{\mathrm{M}}
\newcommand{\Sp}{\mathrm{Sp}}
\newcommand{\Spin}{\mathrm{Spin}}
\newcommand{\Res}{\mathrm{Res}}
\newcommand{\PSp}{\mathrm{PSp}}
\renewcommand{\G}{\mathrm{G}}
\renewcommand{\P}{\mathbb{P}}

\newcommand{\Un}{\mathrm{U}}

\newcommand{\len}{\mathrm{len}}

\newcommand{\rk}{\mathrm{rk}}

\newcommand{\wid}{\mathrm{wd}}
\newcommand{\sA}{\textsc{A}}
\newcommand{\sB}{\textsc{B}}
\newcommand{\sC}{\textsc{C}}
\newcommand{\sD}{\textsc{D}}
\newcommand{\sE}{\textsc{E}}
\newcommand{\sF}{\textsc{F}}
\newcommand{\sG}{\textsc{G}}

\newcommand{\sslash}{\mathbin{/\mkern-6mu/}}

\def\<{\langle}
\def\>{\rangle}
\def\tilde{\widetilde}
\def\o{\tilde{\omega}}
\def\phi{\varphi}
\def \g{\mathcal G}
\def \I{\mathcal I}
\def \W{\mathcal W}
\def \T{\mathcal T}
\def \H{\mathcal H}
\def \X{\mathcal X}
\def \G{\mathcal G}
\def \U{\mathcal U}
\def \S{\mathcal S}
\def \K{\mathcal K}
\def \P{\mathcal P}
\def \E{\mathcal E}
\def \V{\mathcal V}
\def \B{\mathcal B}
\def \EE{\tilde{\mathcal{E}}_\Gamma }
\def\e{\varepsilon}
\def\m{\mathfrak{m}}
\def\n{\mathfrak{n}}
\def\a{\mathfrak{a}}
\def\h{\mathfrak{h}}
\def\d{\mathfrak{d}}
\def\b{\mathfrak{b}}
\def\Or{\mathcal O}
\def\ot{\mathfrak{ o}}
\def\id{\operatorname{id}}
\def\tr{\operatorname{tr}}
\def\Tr{\operatorname{Tr}}
\def\Aut{\operatorname{Aut}}

\def\Hom{\operatorname{Hom}}

\def\w{\tilde{w}}
\def\ga{\tilde{\Gamma}_w^i}

\def \g{\mathcal G}

\title[Geometry of word equations over special fields]
{Geometry of word equations in simple algebraic groups over special fields}

\author[Gordeev, Kunyavski\u\i , Plotkin] {Nikolai Gordeev, Boris Kunyavski\u\i , Eugene Plotkin}

\thanks{The research of the first  author
was financially supported by the Ministry of Education and Science of the Russian Federation, project 1.661.2016/1.4.
The research of the
second and third authors was supported by ISF grant
1623/16 and the Emmy Noether Research Institute for Mathematics.
The paper was written when the second author visited the MPIM (Bonn).
The authors thank all these institutions.}

%\dedicatory{To Boris Isaakovich Plotkin on the occasion of his 90th birthday}

\address{Gordeev: Department of Mathematics, Herzen State
Pedagogical University, 48 Moika Embankment, 191186, St.Petersburg,
RUSSIA} \email{nickgordeev@mail.ru}

\address{Kunyavski\u\i : Department of
Mathematics, Bar-Ilan University, 5290002 Ramat Gan, ISRAEL}
\email{kunyav@macs.biu.ac.il}

\address{Plotkin: Department of Mathematics, Bar-Ilan University, 5290002 Ramat Gan,
ISRAEL} \email{plotkin@macs.biu.ac.il}

\begin{abstract}
This paper contains a survey of recent developments in investigation of
word equations in simple matrix groups and polynomial equations in
simple (associative and Lie) matrix algebras along with some new results on the
image of word maps on algebraic groups defined over special fields:
complex, real, $p$-adic (or close to such), or finite.
\end{abstract}

\maketitle

\begin{epigraph}
{Two youngsters came to a sage with a question: ``One of us thinks that even you feel bad, there is always light at the end of the tunnel, and the other thinks that even things go well, you will be overthrown to hell at some point. Who is right?''

``Both and none'', answered the sage. ``Everything depends on the angle the tunnel constructor has chosen.''}
{{\it Ko Bo Zhen}\footnotemark[1]}
\end{epigraph}

\begin{epigraph}
{The probability of the event that a
randomly picked animal is a panda will be much higher if the samples
that one is allowed to test are restrictively placed within Sichuan
province.}
{{\it B. Gorkin-Perelman}\footnotemark[2]}
\end{epigraph}

\footnotetext[1]{{\it Philosophy for Beginners},``Iron Pagoda'' Publishing House, Kaifeng, 1123.}
\footnotetext[2]{{\it Zoology for Beginners}, ``Yellow River'' Publishing House, Kaifeng, 1923.}

%\begin{keyword} word map \sep polynomial map \sep matrix group \sep
%matrix algebra \MSC 20G15 \sep 20G20
%\end{keyword}

\section {Introduction}\label{sec:intro}
The goal of the present paper is two-fold. First, we give a brief overview of recent
developments in investigation of word equations in simple matrix
groups and polynomial equations in simple (associative and Lie)
matrix algebras. In this respect, it can be viewed as a follow-up to \cite{KBKP},
where an attempt was made to pursue various parallels between group-theoretic
and algebra-theoretic set-ups.

The emphasis is put on the properties of the
image of the word map under consideration. Namely,
ideally we want to prove that this image is as large
as possible, i.e., that the map is {\it surjective} or
at least {\it dominant} (in Zariski or ``natural'' topology).
In the latter case, whenever the surjectivity is unknown,
we are interested in the ``fine structure'' of the image.

Usually, in the most general set-up (arbitrary word maps on arbitrary groups)
little can be said, so one restricts attention to some wide classes of words and groups.
In particular, we are interested in {\it linear algebraic} groups where Borel's dominance
theorem \cite{Bo1} is available for connected semisimple groups. To go further, one can consider some
special classes of words and/or ground fields. The first approach may lead to spectacular
results, see, e.g., our recent paper \cite{GKP3} for a survey.

Here we focus on looking at some special ground fields, such as
complex, real, $p$-adic, finite, or close to such.
(Some recent results valid for arbitrary algebraically closed ground fields were also surveyed
in \cite{GKP3}.)

It is also worth noting
that the case of finite ground fields, which naturally includes
equations in finite groups of Lie type, has been widely discussed in
the literature over the past few years (see, e.g.,
\cite{Sh1}--\cite{Sh3}, \cite{Mall}, \cite{BGK}), mainly in virtue
of spectacular success of algebraic-geometric machinery and solution
of a number of long-standing problems, such as Ore's problem
\cite{LOST1}. Much less is known for matrix equations over number
fields and their rings of integers, as well as over the fields of
$p$-adic, real, and complex numbers (see, however,
\cite{Sh2}--\cite{Sh3}, \cite{AGKS}, \cite{Ku2}). Thus the present paper
contains much more questions than answers, which clearly indicates that
the topic is still in its infancy (if not embryonic) stage.

Our second goal consists in discussing some crucial results in more detail,
providing slightly modified proofs and, more important, giving some generalizations.
We pay special attention to the study of the fine structure of the image, as mentioned above,
with a goal to guarantee that the image contains some ``general'' or ``special'' elements
(regular semisimple, unipotent, etc.). These parts of the paper can be omitted by the reader interested
only in general picture.

Our main message to the reader is encoded in two epigraphs.
After translation into mathematical language, the first says
that when we are looking at the image of a word map $w\colon G^d\to G$
and varying $w$ and $G$, this image can be made as large as possible
(within the constraints determined by the nature of the problem) when
we fix $w$ and enlarge $G$, and, {\it vice versa}, it can be made as
small as possible (also within certain constraints) when we fix $G$
and enlarge $w$ (in the body of the paper we call such a situation ``negative-positive''). The second epigraph can be roughly interpreted as follows:
when we carefully define the class of groups $G$ we consider, any random
word $w$ (if not all of them) has a large image (where ``random'' and ``large''
are also to be carefully defined).

\bigskip

Our notation is standard. We refer the reader to \cite{Seg} for basic notions related to
word maps.

\bigskip

We start with several naive (well-known) examples which will hopefully
give a flavour of problems under consideration. First, consider
extracting square roots in matrix groups.

\begin{example} \label{ex:sl2r}
Is the equation $x^2=g$ always solvable in $G=\SL(2,\mathbb R)$? Of
course, the answer is ``no''. For example, the matrix
$$
g=\left(\begin{matrix}-4 & 0 \\ 0 & -1/4\end{matrix}\right)
$$
has no square roots in $G$: by Jordan's theorem, such a root would
have two complex eigenvalues one of which would be $\pm 2i$ and the
other $\pm i/2$, which is impossible because they must be conjugate.
\end{example}

There are at least two natural ways out. First, one can try to
extend the ground field, going over to $\SL(2,\mathbb
C)$.
%{\bf !! In peevious papers we have used the notation of Springer $G(\C)$ and in theorem below also we have used this notation $\mathcal G(k)$ !!.}
Here one has another counter-example: the matrix
$$
g=\left(\begin{matrix}-1 & 1 \\ 0 & -1\end{matrix}\right)
$$
has no square roots in $\SL(2,\mathbb C)$ because by the same Jordan
theorem, the eigenvalues of such a root would be either both equal
to $i$ or to $-i$, thus giving the determinant $-1$ (and not 1, as
required). This can easily be repaired by factoring out the centre
and considering the {\it adjoint} group $\PSL(2,\mathbb C)$: in this latter group one can extract roots of
any degree (and this can also be done in $\PSL(m,\mathbb C)$ for any
$m\ge 2$).

Surprisingly, this way out is somewhat misleading: it does not work
for simple groups other than those of type $\sA_n$. Here is the
corresponding result:

\begin{theorem} (Steinberg \cite{St4}, Chatterjee \cite{Ch1}--\cite{Ch2}) \label{St}
The map $x\mapsto x^n$ is surjective on $\mathcal G(K)$ ($K$ is an
algebraically closed field of characteristic exponent $p$, $\mathcal
G$ is a connected semisimple algebraic $K$-group) if and only if $n$
is prime to $prz$, where $z$ is the order of the centre of $\mathcal
G$ and $r$ is the product of ``bad'' primes.
\end{theorem}

In particular, one can guarantee that $n$-th roots can be extracted
in an arbitrary connected semisimple group of adjoint type over
$\mathbb C$ if and only if $n$ is prime to 30.

\begin{obs} \label{compact}
Here is another way out of the situation of Example \ref{ex:sl2r}:
replace $\SL (2,\mathbb R)$ with its compact form $\SU (2)$. Then
extracting square roots is no longer a problem. More generally, one
can use Lie theory to extract roots of any degree in any connected
compact real Lie group $G$ because for such a $G$ the exponential
map $\exp\colon\Fg\to G$ is surjective (see, e.g.,
\cite[Corollary~2.1.2]{Do}): indeed, given $g\in G$, write it as
$g=\exp(a)$ and for any integer $n\ge 1$ get $\exp(a/n)^n=g$.

This observation can be put in an even more general form: it turns
out that the surjectivity of the exponential map is equivalent to
the surjectivity of all power maps $G\to G$, $g\mapsto g^n$,
provided $G$ is any connected real \cite{McC}, \cite{HL} or complex
\cite[Section~6]{Ch1} linear algebraic group; more details on the
real case can be found in \cite{DjTh}, \cite{Wu2}, \cite{Ch4}; see
\cite{Ch3} for discussion of similar problems for $p$-adic groups.
The reader interested in the history of the surjectivity problem for
the exponential map, dating back to the 19th century (Engel and
Study), is referred to \cite{Wu1}; see \cite{DH} for a survey of
modern work and \cite{HR} for generalizations to the case of Lie
semigroups.
\end{obs}

Going beyond these examples, one can discuss similar problems for
more general matrix equations. In this paper we restrict our
attention to word equations in a group $G$ of the form
\begin{equation} \label{eq:gr}
w(x_1,\dots ,x_d)=g
\end{equation}
where $w\in F_d$ is an element of the free $d$-generated group (a
group word in $d$ letters), and to polynomial equations in an algebra
$\cA$ of the form
\begin{equation} \label{eq:alg}
P(X_1,\dots ,X_d)=a
\end{equation}
where $P$ is an element of the free $d$-generated associative or Lie
algebra over a field $k$ (an associative or Lie polynomial whose
coefficients are {\it scalars} from $k$). In both cases the
right-hand side is fixed and solutions are sought among $d$-tuples
of elements of $G$ (resp. $\cA$).

This means that if, say., $\mathcal A$ is a matrix algebra, we
consider equations $XY-YX=C $ but {\it not} $BX-XB=C$ or
$AX^2+BX+C=0$. The latter equations are far more difficult,
and the interested reader is referred, e.g., to
\cite{Ge}, \cite{Sl}. As to word equations with constants, see \cite{GKP1}--\cite{GKP3},
\cite{KT} and the references therein (see, however, Section \ref{sec:const} below
for a brief account).

To avoid any confusion, we want to emphasize that in our set-up, solutions of \eqref{eq:gr} are sought
{\it in} $G$, and {\it not in an overgroup} of $G$. The latter option constitutes
a fascinating area of research going back to Bernhard Neumann \cite{Ne}; see \cite[Introduction]{KT}, the survey \cite{Ro} and the references therein for
an overview.

\section{Word equations in groups: surjectivity}

Let $w(x_1,\dots ,x_d)$ be a group word in $d$ letters which is not
representable as a proper power of some other word. For any group
$G$, denote by the same letter the evaluation map
\begin{equation} \label{map:gr}
\w\colon G^d \to G
\end{equation}
defined by substituting $(g_1,\dots ,g_d)$ instead of $(x_1,\dots ,x_d)$
and computing the value $w(g_1,\dots ,g_d)$. We call $\w$ the word map induced by
$w$. Examples \ref{ex:sl2r}
give rise to the following natural questions.

\begin{quest} \label{q:sur}
Let $G=\mathcal G(K)$ where $\mathcal G$ is a connected semisimple
algebraic $K$-group. Is $\w$ surjective when
\begin{itemize}
\item[(i)] $K=\mathbb C$ and $\mathcal G$ is of adjoint type;
\item[(i$^\prime$)] $K=\mathbb R$ and $\mathcal G$ is a split $K$-group of adjoint type;
\item[(ii)] $K=\mathbb R$ and $\mathcal G$ is compact?
\end{itemize}
\end{quest}

Surprisingly, Questions \ref{q:sur}(i), (i$^\prime$) are open, even in the simplest
case $G=\PSL(2,\BC)$, even for words in two letters. Naive attempts
to use Lie theory fail even in the cases where the exponential map
is surjective. Say, the map $\Fg\times \dots\times\Fg\to\Fg$ induced
by a Lie polynomial may not be surjective whereas the ``same'' word
(where each Lie bracket $[X_i,X_j]$ is replaced with the group
commutator $[x_i,x_j]=x_ix_jx_i^{-1}x_j^{-1}$) may induce a
surjective map $G^d\to G$. Here is a concrete example:
$$P = [[[X,Y], X], [[X,Y], Y]]\colon \mathfrak{sl}(2, \BC)\times \mathfrak{sl}(2, \BC)
\rightarrow \mathfrak{sl}(2, \BC)$$ is not surjective \cite{BGKP}
whereas the corresponding map $(x,y)\mapsto [[[x,y], x], [[x, y],
y]]$ is surjective on $\PSL(2,\mathbb C)$ (MAGMA computations in \cite[Section~9]{BaZa}).

There are positive results for some particular words. It is
classically known (\cite{PW}, \cite{Ree}) that under the assumptions of
Question \ref{q:sur}(i), the commutator map is surjective. In the
same setting, the image of any  Engel
word $w=[[x,y],y,\dots ,y]$ contains all semisimple and all
unipotent elements \cite{Go5}. This implies that such words are surjective on
$\PSL (2,\BC )$ (there are different proofs of the latter fact, see
\cite{BGG}, \cite{KKMP}, \cite{BaZa}, \cite{GKP3}).
Some other classes of words in two
variables for which the word map is surjective on $\PSL (2,\mathbb
C)$ were discovered in \cite{BaZa}, see also \cite{GKP1}--\cite{GKP3}.

As to Question \ref{q:sur}(ii), the situation is completely different.

\subsection{Negative-positive results for compact real groups} \label{Thom}
Under the assumptions of Question \ref{q:sur}(ii), most of known results
may be called negative-positive where negative results are obtained
by fixing a group and changing words and positive results, respectively,
are obtained by fixing a word and enlarging groups (see the first epigraph
to the paper).

The main negative-positive result for anisotropic forms of simple algebraic groups over the real field,
that is, connected compact simple  Lie groups (\cite[Ch.~5.2]{VO}), is the following

\begin{theorem} \label{thom-hls}
%{\bf Theorem H. }
$
{}
$
\begin{itemize}
\item[(i)] Let $\G$ be an anisotropic form of a simple linear algebraic group over the real field $\R$, and let $G = \G(\R)$.
Then there exists a non-trivial metric $d(x, y)$ on $G$ such that for any real $\varepsilon >0$ there is a word $w \in F_2$ such that
$$d(1, \w(g_1, g_2)) < \varepsilon$$ for every  $(g_1, g_2) \in G^2$.
\item[(ii)] Let $1\ne w_1(x_1, \dots, x_n)\in F_n, \,\,1\ne w_2(y_1, \dots, y_m)\in F_m, \,\,\,w = w_1w_2$.
Then there exists $c = c(w_1,w_2)$ such that for every simple anisotropic linear algebraic group $\G$ of Lie rank $> c$
and for $G = \G(\R)$ the word map $\w\colon  G^{n+m}\rightarrow G$ is surjective.
\end{itemize}
\end{theorem}

Statement (i) is a theorem of A.~Thom \cite{Th}. Actually, Thom considered $G = \SU_m(\C), w \in F_2$
and $d(x, y) = \left\lVert x - y \right\rVert$ where $\left\lVert \,\,\, \right\rVert$ is the operator norm on the unitary group.
%a norm on complex matrices.
However, for any compact group $G =\G(\mathbb R)$ we may fix a faithful continuous representation
$\rho\colon  G\rightarrow \SU_m(\C)$ and consider the restriction of $d(x, y)$ to $\rho(G)$.
Then we have the corresponding result for $\rho(G)\approx G$ if we consider the restriction of the map
$\w \colon \SU_m(\C) \times \SU_m(\C)\rightarrow \SU_n(\C)$ to $\rho(G)\times \rho(G)$. Also, instead of the operator norm, we can consider 
any unitarily invariant norm, say, the Frobenius norm on the space of square matrices $M_m(\C) \geq \SU_m (\C)\geq G$  
defined by $\left\lVert\{ x_{ij} \}\right\rVert = \sqrt{\sum_{ij} {\left\vert x_{ij}\right\vert}^2}$ 
(it is invariant under left and right multiplication by matrices from $\SU_m(\C)$).

Below we give
a little bit different proof of (i), essentially based on the ideas of \cite{Th}.

\bigskip

\noindent
{\it Proof of (i)}.
%We use the Frobenius norm $\left\lVert\{ x_{ij} \}\right\rVert = \sqrt{\sum_{ij} {\left\vert x_{ij}\right\vert}^2}$ on the space of square matrices $M_m(\C) \geq \SU_m (\C)\geq G$. It is invariant under left and right %multiplication by matrices from $\SU_m(\C)$.
%{\bf UBRAT!!, and for every $\gamma \in \SU_m(\C)$ we have $\left\lVert\gamma \right\rVert  = \sqrt{m}$.}
Let $\left\lVert \,\,\, \right\rVert$ denote a unitarily invariant norm on $G$, and let $d$ denote the induced metric.
For every $g \in G$ let $l(g)\coloneqq d(1, g)$. Then $l(g)\leq c$ for every $g\in G$ where $c\in \R$ is a constant
(because $G$ is a compact group). Standard properties of metric imply that
$$l(hgh^{-1}) = l(g)\,\,\,\text{ and}\,\,\,\, l([g, h]) \leq 2l(g)l(h)$$
for every $g, h\in G$ (see \cite[Lemma 2.1]{Th} for details). %{\bf VRODE OT NORMI ETO NE ZAVICIT? }

The crucial point of the proof in \cite{Th} is the following observation.
For a group $G$ of given Lie rank $r$ and any $\varepsilon \in \R_{>0}$
one can find $q = q(r, \varepsilon)$ such that for every $g \in G$ we have
\begin{equation}
\label{equa110}
l(g^m) < \varepsilon
\end{equation}
for some $1\leq m = m(g) \leq q$.

%{\bf !! Next sentences is not ckear : Indeed, if $l(g^k) > \varepsilon$ for every $k
%\leq q$, then the distance
%between any two elements $g^e, g^f$ with  $e, f\leq q$ is less than $\varepsilon/2$,
%and therefore the Haar measure of the union of $q$ balls of radius $\varepsilon/2$ is %greater than or equal to $\mu = q\cdot $ (the measure of the ball). I propose:}

%\bigskip

Indeed, fix $G$ and $\varepsilon$ and assume to the contrary that for any positive integer $q$  we have $l (g^k) \ge  \varepsilon$ for some $g\in G$ and for all $k \leq q$. Let $e < f \leq q$. Since $\left\lVert gx\right\rVert = \left\lVert x\right\rVert$ for every $g \in G$ and $x \in M_m(\C)$ we get
\begin{equation}
\label{equa111}
d(g^e, g^f) = \left\lVert g^e - g^f\right\rVert =  \left\lVert g^e(1- g^{f-e}) \right\rVert =  \left\lVert 1- g^{f-e} \right\rVert \geq \varepsilon.
\end{equation}
Define
\begin{equation}
\label{equa112}
V_{t, \varepsilon} \coloneqq \{x \in G\,\,\,\mid\,\,\,d(x, g^t) < \frac{1}{2}\varepsilon\}.
\end{equation}
It is a measurable set with respect to the Haar measure $\mu$ on $G$, and for each $t=1,\dots ,q$ we have $\mu(V_{t, \varepsilon}) = \mu(V_{1, \varepsilon}) > 0$. By \eqref{equa112} and \eqref{equa111},
the sets $\mu(V_{t, \varepsilon})$ are disjoint and therefore, for the disjoint union  $V_{q, \varepsilon} = \cup_{t} V_{t, \varepsilon}\subset G$ we have
\begin{equation}
\label{equa113}
\mu(V_{q, \varepsilon}) = \sum_{t = 1}^q \mu(V_{t, \varepsilon}) = q \mu(V_{1, \varepsilon}).
\end{equation}
The measure of the ball of radius $\frac{1}{2}\varepsilon$ is strictly positive and depends on $\varepsilon$. Thus, \eqref{equa113} implies that for a sufficiently large $q$ the subset $V_{q, \varepsilon}\subset G$  will have the measure  which is bigger than any given positive number. This contradicts the compactness of $G$.

\medskip

%{\bf Delete : Then if we fix $\varepsilon$, we can get any $\mu$, which is a %contradiction with the compactness of $G$ (see details in \cite[Lemma~3.2]{Th}).}

\bigskip
%{\bf Change}

Define a sequence of words in $F_2$ by setting
$$w_0 = [ x, y], w_1 = [w_0, x], w_2 = [w_1, y w_1 y^{-1}], \dots, $$$$ w_{2i-1} = [w_{2i-2}, x^i],  w_{2i } = [w_{2i-1}, yw_{2i-1} y^{-1}], \dots $$
It is easy to see that all words in this sequence are non-trivial.

%{\bf Zdes kakaya to putanica c 2}
Fix $\varepsilon >0$. There exists a constant $C> 1$ such that $l(g) \leq C$ for every $g \in G$ (because $G$ is compact). We may assume $\varepsilon <\frac{1}{4C}$.  One can find a positive integer
$q = q(r, \frac{\varepsilon}{C})$ with the following property: for every $g \in G$ there is a positive integer $m = m(g) \leq q$ such that %{\bf HERE
$l(g^m) < \frac{\varepsilon}{2C} $ (see \eqref{equa110}).
 Then for every $h \in G$ we have
$$l (w_{2m-1}(g, h)) = l([w_{2m-2}(g, h), g^m]) \leq 2\underbrace{l(w_{2m-2}(g, h))}_{\leq C}\underbrace{l(g^m)}_{< \e/2C} < \varepsilon, $$
$$l(w_{2m}(g, h) ) = l([w_{2m-1}(g, h), h w_{2m-1}(g,h)h^{-1}])\leq$$
$$\leq 2 l(w_{2m-1}(g, h))^2 < 2\e^2  <  \e \frac{2}{4C}=  \frac{\varepsilon}{2C}. $$
Suppose now that  $l(w_{2k}(g, h)) < \frac{\varepsilon}{2C}$ for some $k \geq m$. Then
$$l (w_{2(k+1)-1}(g, h)) = l([w_{2k}(g, h), g^{k+1}] \leq 2\underbrace{l(w_{2k}(g, h))}_{<  \varepsilon/2C}\underbrace{l(g^{k+1})}_{\leq C} < \varepsilon, $$
$$l(w_{2(k+1)}(g, h) ) = l([w_{2(k+1)-1}(g, h), h w_{2(k+1)-1}(g,h)h^{-1}]) \leq $$$$\leq 2 l(w_{2(k+1)-1}(g, h))^2 < 2\e^2 < \frac{\varepsilon}{2C}. $$

Thus, by induction,
we have $l(w_{2k}(g, h)) < \frac{\varepsilon}{2C}$
for every $g, h \in G$ and for every  $k \geq q$. This proves the  statement.
\qed

\bigskip

Statement (ii) is a theorem of Hui--Larsen--Shalev \cite{HLS}. It can be viewed
as a step towards a conjecture of Larsen (attributed in \cite{ST} to his 2008 AMS talk)
which asserts that any word map is surjective on a connected compact simple real linear
algebraic group $G$ provided its rank is sufficiently large. For certain words, a weaker form
of this conjecture was proved in \cite{ET1} for unitary groups.

Let us give a sketch of proof of (ii) following \cite{HLS}.

\medskip

\noindent
{\it Proof of (ii)}.
Every element of $G$ is contained in $T = \T(\R)$ where $\T$ is  a maximal torus of $\G$ (recall that $\G$ is anisotropic and thus $G$ does not contain unipotent elements). Let $N_G(T)$ denote the normalizer of $T$ in $G$. We have
$N_G(T)/T \approx W$ where $W$ is the Weyl group of $\G$ (see \cite[6.5.9]{GoGr}). Let $\dot w_c\in N_G(T)$ denote a preimage of a Coxeter element $w_c$. 

Recall that if $R$ is an irreducible root system and 
$\Pi = \{\alpha_1, \dots, \alpha_n\}\subset R$ is a fixed set of simple roots, 
a {\it  Coxeter element} of $W$ is any product of reflections $w_c=\prod_i w_{\alpha_i}$ where each $\alpha_i \in \Pi$ appears
exactly once (it is allowed to take reflections $w_{\alpha_i}$ in such a product in any order); see \cite[V.6]{Bou} for details. 

We have $\dot w_c^{-1} = \sigma \dot w_c\sigma^{-1}t_0$ for some $ \sigma \in G$ and $t_0\in T$, and one can show that every element $t \in T$ can be written in the form $t = \dot w_c (s \dot w_c^{-1}s^{-1})$ for some $s \in T$
(see, e.g., \cite{GKP3}).  Hence every element of $Tt_0^{-1} = T$ is contained in the square of the conjugacy class of $\dot w_c$. Note that the image of a word map is invariant under conjugations. Thus, to prove (ii), we have to show that $\dot w_c\in \Imm \w^\prime$ for every non-trivial word $w^\prime \in F_n$
and for the corresponding word map $\w^\prime \colon G^n \rightarrow G$ under the condition that the Lie rank of $\G$ is big enough when $w^\prime$ is fixed.

We may restrict our considerations to the case when $\G$ is of one of the classical types $\sA_r$, $\sB_r$, $\sC_r$, $\sD_r$.
Since the root system $\sD_r$ is a subset of both $\sB_r$ and $\sC_r$, any group
$\G = \G(\C)$ of type $\sB_r$ or $\sC_r$  has a subgroup $\G_1 = \G_1(\C)$ of type $\sD_r$.
Moreover, a maximal compact Lie subgroup $\K_1\leq \G_1$
is also a compact Lie  subgroup of $\G$ and is therefore contained in a maximal compact Lie subgroup $\K$ of $\G$. Let $T$ be a maximal torus of $\K_1$.
Note that $T$ coincides with some maximal torus of $\K$ because $\G$ and $\G_1$ are of the same Lie rank. Every element of $\K$ is conjugate to an element of $T$  which is also a maximal torus of $\K_1$. Therefore, once
we prove that $\w\colon \K_1^{n+m}\rightarrow \K_1$ is surjective, this implies that
$\w\colon \K^{n+m}\rightarrow \K$ is also surjective. Hence we only have to consider the cases $\sA_r$, $\sD_r$.

Let $\G$ be a simple, simply connected group of type $\sA_r$. Then $G = \SU_{r+1}(\C)$.
Consider the word map $\o\colon \SU_2(\C)^n\rightarrow \SU_2(\C)$ for any non-trivial word $\omega \in F_n$.
The image of this map is a connected, compact, non-trivial (being Zariski dense in $\SL_2(\C)$ by the Borel theorem, see Theorem \ref{Borel} below) subset of $G$ containing the identity. The
intersection of a maximal torus $T^\prime$ of $\SU_2(\C)$ and $\o(\SU_2(\C)^n)$ is also a non-trivial compact subset of $T^\prime$ containing $1$. Hence
there is $d$ such that any $t \in T^\prime$ of order $>d$ belongs to $\o(\SU_2^n)$.
Further,  let $r +1 > d$, and let $\xi\colon \SU_2(\C)\rightarrow \SU_{r+1}(\C)$ be an
irreducible unitary representation of $\SU_2(\C)$.
Note that this representation is the restriction to compact subgroups of 
the representation of $\SL_2(\C)$ on binary forms of degree $r$ (see, e.g., \cite[Prop.~4.11]{Hal}).
Denote by $\epsilon_m$ any primitive root of 1 of degree  $m$.  Let 
$$t = \begin{cases} \epsilon_{r+1} \,\,\,\text{if }\,\,\,r+1\,\,\,\text{is odd}\\
\epsilon_{2(r+1)} \,\,\,\text{if }\,\,\,r+1\,\,\,\text{is even}\end{cases}.$$
%Hence if $t$ is a semisimple element of order $r(r+1)$
Then the set of eigenvalues of $\xi(t)$ consists of all roots $\sqrt[r+1]{1}$ if $r+1$ is odd, and of  all roots $\bf \sqrt[r+1]{1}$ multiplied by a fixed
root $\epsilon^r_{2(r+1)}$ if $r+1$ is even. One can find a preimage $\dot w_c\in \SU_{r+1}(\C)$ of a Coxeter element $w_c$ which has such a set of eigenvalues. 
%The eigenvalues of any Coxeter element $w_c$ of $\SU_{r+1}(\C)$ are also all roots $\sqrt[r+1]{1}$
(Note that a Coxeter element of $\SU_{r+1}(\C)$ corresponds to a monomial matrix of cyclic permutations of  an orthogonal basis.) Then $\xi (t)$ is conjugate to $\dot w_c$ in $\SU_{r+1}(\C)$.
 Indeed, both matrices
are unitary and have the same set of eigenvalues.
%diagonal, with the same diagonal entries, hence by the theorem
%on unitary diagonalization they are conjugate in $\mathrm U(r+1,\C)$; after scaling
%a conjating matrix $P$ by the $(r+1)$-th root of $\det(P)$, we conclude that they are
%conjugate in $\SU_{r+1}(\C)$.

Now consider non-trivial word maps  $$\w\colon
\SU_{r+1}(\C)^n\rightarrow \SU_{r+1}(\C),\,\,\,\o\colon \SU_2(\C)^n\rightarrow \SU_2(\C)$$   which correspond to the same word $w$. The diagram
$$\SU_2(\C)^n\stackrel{\o}{\rightarrow} \SU_2(\C)$$
$$\downarrow \xi^n \,\,\,\,\,\,\,\,\,\downarrow \xi$$
$$\SU_{r+1}(\C)^n\stackrel{\w}{\rightarrow} \SU_{r+1}(\C),$$
where $\xi^n ((g_1, \dots, g_n)) \coloneqq (\xi(g_1), \dots, \xi_(g_n)$, is commutative because both $\xi$ and $\xi^n$ commute with word maps. Then, if we have
$\dot w_c $ in $\Imm \w\circ \xi$, we also have $\dot w_c\in \Imm \w$.
Thus we get our statement for the case $\sA_r$. The case $\sD_r$ is treated by similar arguments, see \cite[Section~2]{HLS} for details. 
%{\bf see, \cite{HLS}} (???)
\qed

\begin{remark} \label{q:local}
%Apparently, one can extend Thom's construction
Let $\cG$ be  an arbitrary anisotropic simple group defined over a non-archimedean local field $k$
(which is necessarily of type $\sA_n$). Recall that by the Bruhat--Tits--Rousseau theorem
(see \cite{Pr} for a short proof), $\cG$ is anisotropic if and only if
$G=\cG(k)$ is compact in the topology induced by the valuation of $k$.
We have $G=\SL(1,D)$, the group of elements of reduced norm 1 of a division $k$-algebra $D$. Moreover, there exists a series $\{G_i\}_{i=0}^\infty$ of normal subgroups  $G_i \lhd G$ such that
$$G_0 = G, \,\,\,[G_0, G_0] = G_1,\,\,\, [G_1, G_i] \leq G_{i+1}, \dots$$
with 
$$G_i \subset 1 + \mathfrak{P}^i_D, \,\,\,\,\text{where}\,\,\, \mathfrak{P}^i_D = \{x \in D\,\,\,\mid\,\,\,v_D(x) \geq i\}$$
(here $v_D(x) = \frac{1}{c}v_p(\mathrm{Nrd}_{D/k}(x))$ is the non-archimedean discrete valuation on $D$ induced by the  non-archimedean discrete valuation $v_p$ on $k$, $c$ is the index of $D$, 
$\mathrm{Nrd}_{D/k}$ is the reduced norm; see \cite{Ri}, \cite[1.4]{PR}). Let now $\left\lVert x \right\rVert_p \coloneqq p^{- v_D(x)}$ be the corresponding norm on $D$. 
Since $\mathrm{Nrd}_{D/k}\colon G\to k^*$ is a group homomorphism, the norm  $\left\lVert  \,\,\, \right\rVert_p$ is invariant with respect to left and right multiplication 
by elements of $G$.  
Further, let $F_n$ be the free group of the rank $n$, and let $$F_n^0 \coloneqq F, F_n^1 \coloneqq [F_n^0, F_n^0], \dots , F_n^i \coloneqq [F_n^1, F_n^{i-1}], \dots $$ 
Then for every $w \in F_n^i$ and every $(g_1, \dots, g_n) \in G^n$ we have
$$\left\lVert\w(g_1, \ldots, g_n) - 1 \right\rVert _p \leq p^{-i}.$$
%Any element $g\in G$ lies in some maximal torus of $G$, which consists of the elements
%of norm 1 of some field extension $L/k$. For the group of such elements, the group $U_1$ of principal units of $L$,
%we have ${\lVert 1-g^{p^l}\rVert}_p \to 0$, where $p$ denotes the residue characteristic of $k$
%(see, e.g., \cite[5.6, 5.7, 6.1]{FV}). Hence we have an analogue of the key inequality \eqref{equa110}.
Thus Thom's phenomenon can also be observed for simple anisotropic groups over non-archimedean local fields.   
\end{remark}

\begin{remark}
Thom's phenomenon has been further
investigated in \cite{ABRdS}, \cite{ET2} where it got a name
of ``almost law'' in $G$.
\end{remark}

In this setting, there are also some positive results for particular
words:
\begin{itemize}
\item any Engel word is surjective on any compact $G=\G(\R)$ (\cite{ET1} for $\SU (n)$, \cite{Go5} in general);
\item if $w\in F_2$ does not belong to the second derived subgroup $F_2^{(2)}$, then for infinitely many $n$
the induced word map is surjective on $\SU (n)$ \cite{ET1}.
\end{itemize}

\subsection{Non-compact real groups}
Little is known here.
The following question seems the most
challenging.

\begin{quest}
Can one observe the phenomenon of ``almost laws'' in a {\em
non-compact} simple linear algebraic $\BR$-group $\cG$? Say, in a
{\em split} $\BR$-group? More precisely, let $G=\cG(\BR)^0/Z$
be the identity component of the group of real points of $\cG$
modulo centre. ($G$ is simple, see, e.g., \cite[Section~3.2]{PR}.)

Does there exist a non-power word $w$ ($w \ne v^k$, $k>1$)
inducing a non-surjective map $\w\colon
G\times\dots\times G\to G$?
\end{quest}

Even the case $\cG=\SL_2$ is open. We can only prove the following simple assertion,
which is a generalization of a result from \cite{HLS}.

\begin{prop} \label{prop:PSL2R}
Let $G = \PSL_2(\R)$, and let $w\in F_d$ be any nontrivial word.
Then the image of the word map $\w\colon G^d \rightarrow G$ contains all split semisimple elements.
Moreover, if $\Imm\, w$ contains an involution, then $\Imm\, w $ contains all semisimple elements of $G$.
\end{prop}
\begin{proof}
Note that for $d = 1$ the statement obviously holds. Further, we need the following fact, which generalizes an assertion from \cite[proof of Theorem~3.1]{HLS}.
\begin{lemma}
\label{lem972}
Let $L$ be any infinite field (not necessarily of characteristic zero), and let
$\o\colon \SL_2(L)^n\rightarrow \SL_2(L)$ be the word map corresponding to
a non-trivial word $\omega \in F_n$. Then there exists a non-constant polynomial $\Phi(x, y) \in L[x,y]$  such that $\Phi(0, 0) = 2$ and
$$ \Phi (\alpha, \beta)\in \Imm\, \tr\circ \o\,\,\,\,\text{for every}\,\,\,\alpha , \beta \in L.$$
\end{lemma}

\begin{proof}
Let $(g_1, g_2, \dots, g_n) \in \SL_2(L)^n$. We may assume $\omega(1, g_2, \dots, g_n) =1$ for every $g_2, \dots, g_n$ (otherwise we may reduce our consideration to the case of the word in $n-1$ variables).

Now, fix the elements $g_2, \dots, g_n$ and take the element $g_1$ of the form
\begin{equation}
\label{equa555}
g_1 =
\begin{pmatrix} 1 & y\cr
x&1+xy\cr \end{pmatrix},\,\,\,x, y\in L.
\end{equation}
Then
$$g_1^{-1}=\begin{pmatrix} 1+xy & -y\cr
-x&1\cr \end{pmatrix}.
$$
Hence  $\tr \o (g_1,g_2, \ldots, g_n) = \Phi(x, y)$ is a  polynomial in two variables $x,y$ over a field $L$.
Suppose that for every fixed $g_2, \dots, g_n \in \SL_2(L)$ we have
$\Phi(x,y)\equiv c$, a constant polynomial. Then $c=2$ for every $g_1$
because $$\Phi(0, 0)= \tr (\o(1, g_2, \dots, g_n) = \tr 1  = 2.$$
Since every non-central element of $\SL_2(L)$ is conjugate to an element of the form (\ref{equa555}) (see \cite{EG1}),
the equality $\tr \w(g_1, g_2, \dots, g_n)=2$ for every $g_1, g_2, \dots, g_n \in \SL_2(L)$,  where $g_1$ is an element of the form (\ref{equa555}), implies the equality $\tr \w(g_1, g_2, \dots, g_n)=2$ for every $g_1, g_2, \dots, g_n \in \SL_2(L)$. Thus,
the image of
$\o\colon\SL_2(L)^n\rightarrow \SL_2(L)$ consists of unipotent elements. Since $\SL_2(L)$ is Zariski dense in $\SL_2(\overline{L})$ (where $\overline{L}$ is the algebraic closure of $L$) \cite[18.3]{Bo2}, the image of $\o\colon\SL_2(\overline{L})^n\rightarrow \SL_2(\overline{L})$ also consists of unipotents elements, which contradicts Borel's dominance theorem (see Theorem \ref{Borel} below).
Hence there are  elements $g_2, \dots, g_n\in \SL_2(L)$ such that
$$\Phi(x, y) = \tr \o\left(\begin{pmatrix} 1 & y\cr
x&1+xy\cr \end{pmatrix}, g_2, \dots, g_n\right)$$
is a non-constant polynomial.
\end{proof}

We also use the following well-known lemma.
\begin{lemma}
\label{lem555}
Let $g \in \SL_2(\R)$  be a semisimple element, $g \ne \pm 1$.
It is split if and only if $\lvert\tr g\rvert >  2$. It is of order $4$ if and only if $\tr g = 0$.
\end{lemma}
\begin{proof}
If $g \in \SL_2(\R)$, then either it belongs to a split torus
and is then conjugate to
$$\begin{pmatrix} \alpha & 0\cr
0&\alpha^{-1}\cr \end{pmatrix}, \,\,\, \alpha \in \R^*,$$
or it belongs to an anisotropic torus and is then conjugate to
$$\begin{pmatrix} \cos \phi & \sin \phi\cr
-\sin\phi&\cos\, \phi\cr \end{pmatrix}, \,\,\, \phi \in \R.$$
In the first case
$$\left\vert \tr g \right\vert = \left\vert \alpha + \alpha^{-1}\right\vert \geq 2. $$
In the second case
$$\left\vert \tr g\right\vert = 2\left\vert \cos\, \phi\right\vert \leq 2.$$
Moreover,
$$ \left\vert \tr g\right\vert = 0 \Leftrightarrow \cos\, \phi = 0 \Leftrightarrow \text{the order of}\,\,\,g \,\,\,\text{is equal to}\,\,\,4.$$
%Further, $$\bf \alpha + \alpha^{-1} =2 \cos\, \phi \Rightarrow \alpha + \alpha^{-1} = 2 \cos\, \phi = \pm 2 \Rightarrow
%\alpha = \pm 1, \phi =  \pi m, 2\pi m.$$
%If $\tr\, g = \alpha +\alpha^{-1}$ and $g$ belongs to the diagonal group, then
%$$g = \begin{pmatrix} \alpha & 0\cr
%0&\alpha^{-1}\cr \end{pmatrix}\,\,\,\text{or}\,\,\,g = \begin{pmatrix} \alpha^{-1} & 0\cr
%0&\alpha \cr \end{pmatrix} = \begin{pmatrix} 0 & -1\cr
%1&0\cr \end{pmatrix}\begin{pmatrix} \alpha & 0\cr
%0&\alpha^{-1}\cr \end{pmatrix}\begin{pmatrix} 0 & 1\cr
%-1&0\cr \end{pmatrix},$$
%and if $\tr\, g = 2\cos \phi$ and $g$ belongs to the rotation subgroup, then
%$$g = \begin{pmatrix} \cos \phi & \sin \phi\cr
%-\sin\phi&\cos\, \phi\cr \end{pmatrix}\,\,\,\text{or}\,\,\,g = \begin{pmatrix} \cos \phi & -\sin \phi\cr
%\sin\phi&\cos\, \phi\cr \end{pmatrix} = $$$$= \begin{pmatrix} 0 & -1\cr
%1&0\cr \end{pmatrix} \begin{pmatrix} \cos \phi & \sin \phi\cr
%-\sin\phi&\cos\, \phi\cr \end{pmatrix}\begin{pmatrix} 0 & 1\cr
%-1&0\cr \end{pmatrix}.$$
\end{proof}

Consider now the word map $\w\colon \SL_2(\R)^d \rightarrow \SL_2(\R)$ corresponding to the same word $w$ (we also denote it by $\w$).
We  may assume that $w(1, g_2, \dots, g_d) =1$.

Further, let $\Phi \in \R[x, y]$ be a polynomial satisfying the condition of Lemma \ref{lem972} (here $L =\R$).
Note that the set of values of a non-constant real polynomial consists either of all real numbers,
or of all real numbers $\geq r$,  or of all real numbers $\leq r$ for some $r \in \R$.
Since $2 = \tr \w(1, g_2) = \Phi(0, 0)$,
either all elements $g \in \SL_2(\R)$ with $\tr g \geq 2$, or all elements with $\tr g \leq  2$ belong  to
the image of $\w\colon \SL_2(\R)^d\rightarrow \SL_2(\R)$ (see Lemma \ref{lem972}). Since for every split semisimple element $g$
of $\SL_2(\R)$
we have $\tr g \geq 2$ or $\tr (-g) \geq 2$ (Lemma \ref{lem555}), every split semisimple element of $G = \PSL_2(\R)$ belongs to the image
of the map $G^d \rightarrow G$.  Suppose now that there is an element of order $4$ in the image of $\w\colon\SL_2(\R)^d\rightarrow \SL_2(\R)$ (obviously, this is equivalent to the existence of an element of order $2$ in the image of the word map
$\w\colon \PSL_2(\R)^d\rightarrow \PSL_2(\R)$). Then, according to Lemmas \ref{lem972} and  \ref{lem555},
either all elements $g \in \SL_2(\R)$ with $\tr g \geq 0$ or all elements with $\tr g \leq  0$
belong to the image of the map $\w\colon\SL_2(\R)^2\rightarrow \SL_2(\R)$  and therefore all semisimple elements belong to  the image of the map
$\w\colon \PSL_2(\R)^d\rightarrow \PSL_2(\R)$.
\end{proof}

\begin{remark}
The difference between the compact and noncompact cases may turn out
to be essential also at the level of eventually applicable
techniques. For example, in the compact case one can try to detect
the non-surjectivity of the word map by homological methods. Indeed,
denote $M=\cG(\BR)\times\dots\times\cG(\BR)$, $N=\cG(\BR)$, $m=\dim
_{\BR}(N)$, and assuming that $N$ is compact, consider the induced
map of homology groups $w^*\colon H_m(M)\to H_m(N)$ (the
coefficients may be arbitrary because $M$ and $N$ are orientable as
any Lie group). If $w^*$ is a nonzero map, then $\w$ must be a
surjective map: otherwise it could be factored through
$N'=N\setminus\{\text{\rm{point}}\}$. This would lead to a
contradiction: $H_m(N')=0$ because $N'$ is not compact (see, e.g.,
\cite[Proposition~3.29]{Hat}). Apparently, this approach may only work
in the compact case when $H_m(N)\ne 0$ (see, e.g.,
\cite[Theorem~3.26]{Hat}). (We thank E.~Shustin for this observation.)

See \cite{KT} for alternative approaches of topological nature.
%{\bf A PRIMERI TUT EST? NAPRIMER w = [x, y]?}
\end{remark}

Further, assuming that Question \ref{q:sur} is answered in the
negative, one can ask whether there are obstructions to the
surjectivity detectable at the level of real points.

\begin{quest}
Let $\cG$ be a connected simple linear algebraic $\BR$-group of adjoint type.
Let $G=\cG(\BR)^0$ be the identity component of the group of real points.
Does there exist a non-power word $w$ ($w \ne v^k$, $k>1$) such that the map
$G\times\dots\times G\to G$ is surjective but the
map $\cG(\BC)\times\dots\times\cG(\BC)\to\cG(\BC)$ is not?
\end{quest}

Note that for power words the situation of this question can arise:
say, look at $w=x^2$ and $\cG$ a compact form of a simple group of
type B, C or D. Then the squaring map is surjective on $G$
(see Observation \ref{compact}) but not on $\cG(\C)$ (see Theorem \ref{St}).

\section{Word equations with general right-hand side}

As Question \ref{q:sur}(i) is still unanswered and Question
\ref{q:sur}(ii) is answered in the negative, one has to decide how
to modify the approach to equation \eqref{eq:gr}. In this
connection, let us quote \cite[Principle~2.18]{Ku1} (rechristening
it and hoping that the reader will excuse self-citation):

\begin{prin}
A reasonable property of a reasonable mathematical object lying
inside a reasonable class of objects may not hold but it will hold
at least for an object in general position (if not always), provided
the class under consideration is enlarged or restricted, if
necessary, in an appropriate way.
\end{prin}

In even more loose terms, this principle is formulated in the second epigraph to the paper.

%this means that the probability of the event that a
%randomly picked animal is a panda will be much higher if the samples
%that one is allowed to test are restrictively placed within Sichuan
%province.

\begin{remark}
In the set-up under consideration, the spirit of this principle
consists in solving equation \eqref{eq:gr} for a ``general'' element
$g$ of the group $G$, when $G$ either runs through the same class of
groups, namely, the class of (rational points of) simple linear
algebraic groups of adjoint type (so we stay within Sichuan
province), or through some larger class (so we try to extend the
areal).

Certainly, the problems become meaningful only after one makes the
term ``general'' (or similar often used euphemisms, such as
``generic'', ``random'', ``typical'', and the like) into some
precisely defined notion. Note that the answer to the relevant
questions may heavily depend on the choice of such a definition.
There are lots of possibilities, and we are not going to discuss
them in this paper, referring the reader, say, to the papers of
M.~Gromov \cite{Gr1}, \cite{Gr2}, A.~Ol'shanski\u\i \ \cite{Ols},
Y.~Ollivier \cite{Oll}, I.~Kapovich and P.~Schupp \cite{KaSc1},
\cite{KaSc2}, N.~M.~Dunfield and W.~P.~Thurston \cite{DuTh},
M.~Jarden and A.~Lubotzky \cite{JL}, Y.~Liu and M.~M.~Wood \cite{LW},
etc., for comparing different approaches to randomness in groups.

Anyhow, we cannot avoid mentioning the only general result of this
flavour, a theorem of A.~Borel.

\end{remark}

\begin{theorem} \cite{Bo1} \label{Borel}
If $K$ is a field, $\G$ is a connected semisimple linear algebraic
$K$-group, and $w\ne 1$, then the corresponding word map $\w\colon
\G^d\to \G$ is {\em dominant}.
\end{theorem}

Recall that this means that the image of the map contains a Zariski
dense open set (i.e., for  a ``typical'' right-hand side equation
\eqref{eq:gr} is solvable).

This result has a nice consequence: if $\G$ and $w$ are as in Borel's theorem 
and $K$ is algebraically closed, 
the word width of $G=\G(K)$ is at most two, i.e., every $g\in G$ can be represented as a
product of at most two $w$-values.

\begin{remark}
Bringing Borel's theorem together with Thom's example, one
immediately convinces oneself that the panda principle formulated
above is to be refined: the answer to the question whether panda is
a typical animal in Sichuan may depend on what is meant by
``typical''. Indeed, Thom's example shows that for some word $w$ all
pandas (=unitary matrices from the image of $w$) live within an
$\varepsilon$-neighbourhood of 1, so Thom would not call them
typical. However, Borel probably would: $\varepsilon$-neighbourhood
is Zariski dense!
\end{remark}

\begin{remark}
In the spirit of negative-positive results mentioned in the previous section,
one can hope that the image of any word map on a compact group $G$ is large
provided the Lie rank of $G$ is sufficiently large. More concretely, we would like to mention
the following density statement, which can be viewed as a metric analogue of Larsen's
conjecture.

Given $\e>0$, a subset $Y$ of a metric space $X$ is called $\e$-dense if
the distance from any point $x\in X$ to $Y$ is at most $\e$. Let $G=\SU(n)$,
and let $d_{\text{\rm{rk}}}(g,h):=(\rk (g-h))/n$ denote the normalized
rank metric. J.~Schneider and A.~Thom \cite{ST} proved that given $\e>0$ and a non-trivial
word $w\in F_d$, there exists an integer $N$ depending on $\e$ and $w$ such that
the image of the word map $\w\colon \SU(n)^d\to\SU(n)$ is $\e$-dense in normalized rank metric
for all $n\ge N$.
\end{remark}

Let us now ask what happens outside Sichuan and try to extend
borders.

First note that over-optimistic attempts may fail, in the sense that
the image of a ``typical'' word map is ``not so large''.
To make this vague statement a little more precise, it is convenient to
make use of the notion of width.

\begin{definition} \label{wid}
Let $G$ be a group, and let $w\in F_d$ be a word. For any $g\in G$ define
its $w$-length $\ell_w (g)$ as the smallest $k\in \mathbb N\cup\infty$ such that $g$ can
be represented as a product of $k$ values of $\w\colon G^d\to G$.

The $w$-width of $G$ is defined by $\wid_w(G)\coloneqq \sup_{g\in G}\ell_w(g)$.
\end{definition}

With this notion in mind, one can roughly estimate how large is the image of
a word map on a group $G$ in the situation where the surjectivity or dominance
fail to hold (or are unknown to hold, or the dominance makes no sense): informally,
smaller is the $w$-width of $G$, larger is the image of $\w\colon G^d\to G$.

The first result to be mentioned here is a theorem
of A.~Myasnikov and A.~Nikolaev \cite{MyNi}: for any $w$, any
(non-elementary) hyperbolic group has infinite $w$-width. According
to A.~Ol'shanski\u\i \ \cite{Ols}, hyperbolic groups are ``generic'' within the
class of all groups, so typically a group will have infinite word
width.

Let us make a more modest attempt. Say, in Borel's theorem let us
try to replace ``algebraic group'' with ``Lie group''. Then the
assertion on word width mentioned above may break down. Indeed, let
$w=[x,y]$ be the commutator. Then another theorem of A.~Borel
prevents from far-reaching generalizations:

\begin{theorem} \cite{Bo3}
Let $G$ be a connected semisimple Lie group. Then $G$ has finite
commutator width if and only if its centre is finite.
\end{theorem}

In particular, the universal cover $\widetilde{\SL (2,\mathbb R)}$
of $\SL (2,\mathbb R)$ has infinite commutator width (this
observation is attributed to J.~Milnor, cited from \cite{Wo}).

\begin{remark} \label{rem:dif}
Let us make another attempt, insisting on the {\it simplicity} of
$G$. There are simple groups $G$ of infinite commutator width
(J.~Barge and E.~Ghys \cite{BG} (infinitely generated), A.~Muranov
\cite{Mu} (finitely generated), P.-E.~Caprace and K.~Fujiwara
\cite{CF} (finitely presented), E.~Fink and A.~Thom \cite{FT}
(with finite palyndromic width). There are also examples of groups
$G$ for which $\wid_w(G)\in \mathbb N$ can be made arbitrarily
large by varying $w$ (see \cite{Mu} and Section \ref{sec:Waring} below).
In the latter case such examples can be obtained from Theorem
\ref{thom-hls}(i). It is interesting whether such an example exists
among simple compact algebraic groups over a {\it non-archimedean} local field.
A general result of A.~Jaikin-Zapirain \cite{JZ} indicates that in such
groups the $w$-width is finite for any non-trivial $w$ but does not
say whether it can be arbitrarily large. In this connection, see Question
\ref{q:local}.

Geometric ideas of \cite{BG} were
further developed to produce more examples of similar flavour, see,
e.g., \cite{GaGh}. However, there are also several classes of simple
groups naturally appearing in topological context (see, e.g.,
\cite{Ts2}) where every element is a commutator. It would be
interesting to pursue investigation of more general word maps on
such groups, especially in view of their relationship with deep
geometric properties of groups under consideration. We refer the interested reader
to \cite{BIP}, \cite{Ts1}, \cite{CZ}, \cite{LaTe} and the references therein.
\end{remark}

\begin{remark}
A little more successful attempt concerns a generalization of Borel's dominance theorem
from {\it semisimple} to {\it perfect} linear algebraic groups \cite{GKP3}. Recall that a group
is said perfect if it coincides with its commutator subgroup. Let $K=\mathbb C$
(or, more generally, any algebraically closed field of characteristic zero).
Let $\g$ be a perfect $K$-group, and let $G=\g(K)$. We identify $G$ with $\g$.
Denote by $U$ the unipotent
radical of $G$, then $G/U$ is a semisimple algebraic $K$-group \cite[11.21]{Bo2}.
By Mostow's Theorem \cite{Mo} (see, e.g., \cite[Th.~VIII.4.3]{Ho},
\cite[Prop.~5.4.1]{Co} for modern exposition), there exists a closed linear algebraic
subgroup $H$ of $G$ (called a Levi subgroup) isomorphic to $G/U$. (Equivalently, $G=HU$ is a semidirect product.) All Levi subgroups
are conjugate. We fix one of them and denote by $H$ throughout below.

Let
$$U_1 = U,    \, U_2 = [U, U_1], \dots , U_i = [U,  U_{i-1}],\dots ,U_{r+1} =\{1\}
$$
be the lower central series of $U$, and let
$V_i = U_i/U_{i+1}$ denote its quotients.
Then we may view $V_i$ as a $K[H]$-module (indeed, the action of $H$ on $V_i$
induced by conjugation of $U$ by elements of $G$ is $K$-linear because $\ch \,K = 0$).

We say that a $K[H]$-module $M$ is {\em augmentative} if it has no $K[H]$-quotients $M/M^\prime$ on which
$H$ acts trivially. If $G$ is a perfect group, $V_1$ is an augmentative $K[H]$-module \cite{GoSa}, \cite{Go3}.

We say that $G$ is a {\em firm} perfect group if $V_i$ is an augmentative $K[H]$-module for every $i$.
(If the nilpotency class of $U$ is equal to one, that is, if $U$ is an abelian group, then any perfect group $G$ is firm.)

We say that $G$ is a {\it strictly firm} perfect group if for every $i$ the space $V_i$ has no nonzero $T$-invariant vectors
(here $T$ denotes a maximal torus of $G$).

Then we have the following analogue of Borel's theorem \cite{GKP3}:

\begin{itemize}
\item[(i)] If $G$ is strictly firm, then for any non-trivial $w \in F_d$  the map
$\w\colon G^d\rightarrow G$ is dominant.
\item[(ii)] If $G$ is firm, then for any $w = w_1(x_1, \dots, x_n)w_2(y_1, \dots, y_k)\in F_{n+k}$, $w_1, w_2 \ne 1$,
the map $\w\colon G^{n+k}\rightarrow G$ is dominant.
\end{itemize}

\end{remark}

It would be interesting to treat the case of perfect groups up to the end.

\begin{quest}
Do there exist a connected perfect $K$-group $\g$ and a non-identity word
$w\in F_d$ such that the word map $w\colon (\g (K))^d\to \g(K)$ is not dominant?
\end{quest}

\begin{remark} \label{rem:Cr}
In a similar spirit of extending borders, one can turn to the
Cremona group $G_0=\Cr (2,K)$ (the group of birational automorphisms
of the projective plane $\mathbb P_K^2$), where $K$ is an algebraically closed
field (say, $K=\BC$). In many respects, $G_0$ is similar to simple
linear algebraic groups (cf. Serre \cite{Ser1}, \cite{Ser2}).
%Similar questions can be posed for the real Cremona group $\Cr
%(2,\BR )$. As explained to the author by J.~Blanc, for this group
%analogous structure information can be deduced from \cite{CL},
%\cite{Bl} and \cite{BM}.
It is also a
good candidate for studying word maps for the following reason.
Although it is not simple as an abstract group (\cite{CL} for
$K=\mathbb C$, \cite{Lo} for an arbitrary $K$), it is simple as a
topological group with respect to several natural topologies: Blanc
\cite{Bl} showed this for the Zariski-like topology introduced by
Serre \cite{Ser2}, and Blanc and Zimmermann \cite{BlZi} treated the
case of a local field $K$ and Euclidean topology (introduced in
\cite{BlFur}). Since in the latter case $G_0$ may not be even perfect
(see \cite{Zi} for the case $K=\mathbb R$), to be on the safer side,
we put $G \coloneqq [G_0,G_0]$.
\end{remark}

The following natural questions arise.

\begin{question} \label{q1:Cr}
Let $w\in F_d$ be a non-identity word.
\begin{itemize}
\item[(i)] Is the map $\w\colon G^d\to G$ dominant in the Zariski
topology?
\item[(ii)] Let $K$ be a local field. Is the map $\w\colon G^d\to G$ dominant in the
Euclidean topology?
\end{itemize}
\end{question}

As to the surjectivity problem, one cannot be over-optimistic in
view of the case of power words. Say, if $K$ is finite, the orders
of elements of $G$ are bounded (see \cite{Ser1} for details), and
thus there are non-surjective power maps. Moreover, this observation
extends to the case where $K$ is algebraically closed: in this case
$G$ contains elements $g$ that are not infinitely divisible (such
are all elements of infinite order not conjugate to elements of $\GL
(2,K)$), and hence there are power maps whose image does not contain
$g$; see \cite{MO1} for details (due to J.~Blanc).

However, for non-power words the following question is meaningful.

\begin{question} \label{q2:Cr}
Let $w\in F_d$ be a non-power word. Can the map $\w\colon G^d\to G$
be non-surjective?
\end{question}

The authors would not be too much surprised if the spectacular results
cited above could help, on the one hand, in finding a non-trivial
word $w$ inducing a non-surjective map, and, on the other hand, in
proving theorems of Borel flavour.

More generally, one can ask the following question.

\begin{question}
Do there exist a locally compact topological group $G$, simple at least as a
topological group, and a word $w = w(x_1, \dots ,x_d)$ non-representable as a proper power of
another word, such that the corresponding word map $\w\colon G^d \to  G$  is not surjective but the
image of $\w$ is dense?
\end{question}

\section{Fine structure of the image of a word map}

In this section we consider the situation where the surjectivity of the word map
$\w\colon G^d\to G$ is not known, and we are looking for subtler features
of the image of $\w$. In particular, we search for
elements of certain type: semisimple (desirable in abundance) or unipotent. These cases are totally
different and require different methods.

We start with the case of groups of  Lie rank 1, which is in fact crucial for what follows.

\subsection{Search for semisimple elements in groups of Lie rank 1}
Let $H = \SL_2(L)$ where $L\subset K$ is an infinite subfield of an algebraically closed field $K$.
Since $\SL_2$ is a connected reductive group, $H$ is dense
in $G = \SL_2(K)$. Thus, $\w(H^d)$ is dense in $\w(G^d)$.

Then, according to \cite{BaZa} (see also \cite{GKP3}), the set $\w(H^d)$ contains an infinite
set of representatives of different semisimple conjugacy classes of $G$.
The latter fact has been proved (by a different method) and used in \cite{HLS}.
Also in \cite{HLS} it has been proved that the set $\w(H^d)$
contains  an infinite set of representatives of different split semisimple conjugacy classes of $G$ if $\R \subset  L$ or $\Q_p\subset L$.
Here we give a generalization of this result.

First of all let us define a class of fields we will consider.
\begin{definition}
A field is called {\em quadratically meagre} if it admits only finitely many different quadratic extensions.
\end{definition}

%\begin{remark}
Note that both $\R$ and $\Q_p$ are quadratically meagre fields. In the case where
the ground field is $\R$, Proposition \ref{prop:PSL2R} guarantees that all
{\it split} semisimple elements of $\PSL_2(\R)$ belong to the image of every non-trivial word map. It is natural to try to generalize this fact to other ground fields.
%\end{remark}

\begin{remark}
Let $F$ be a quadratically meagre field of characteristic zero. Then there is a finite set of primes $S_F^\prime= \{p_1, \dots, p_r\}$
such that if $p\notin S_F^\prime$, then $\sqrt{p}\in F$.
\end{remark}

Let $p_\infty$ denote the archimedean place of $\Q$, and define $S_F= S_F^\prime\cup \{p_\infty\}$.

%Recall that $K$ here is an algebraically closed field. Let $L\leq K$ be a subfield.

\begin{theorem}
\label{th597}
Let $L$ be a field of characteristic zero which contains a  quadratically meagre subfield. Further, let $G = \SL_2(L)$, and let $\w\colon G^d\rightarrow G$
be the word map induced by a non-trivial word $w\in F_d$.
Then $\w(G^d)$ contains an infinite set of representatives of different split semisimple conjugacy classes of $G$.
\end{theorem}
\begin{remark}
It is a well-known fact that the conjugacy class of a  split semisimple element is $\SL_2(F)$ is uniquely determined by the value of the trace.
\end{remark}
\begin{proof}
Partially we follow the ideas of the proof of Lemma 3.2(ii) of \cite{HLS}.
%In particular, we use the following fact proved in \cite{HLS}. {\bf In characteristic zero??.}
%(it can also be derived from the proof of Theorem D.)

%\begin{lemma}
%\label{lem397}
%\bf Let the group  $\Gamma = \SL_2$ is defined and split over an infinite field $F$ (not necessary of the characteristic zero) and let $\omega \in F_n$ be a non-trivial word. Then
%there is an $F$-defined morphism  {\bf of varieties !!!}$\phi\colon \mathbb G_a\rightarrow \Gamma^n$ such that $f = \tr \circ \w\circ \phi$ is a non-constant polynomial on $\mathbb G_a(K) = K^+$.
%\end{lemma}
%\begin{proof}
%\bf By the same argument which was used in the proof of  Proposition \ref{prop:PSL2R} one can costruct the map
%\begin{equation}
%\label{equa1001}
%\psi : \{(\begin{pmatrix} 1&x\cr y&1+xy\cr \end{pmatrix}, \gamma_2, \ldots, \gamma_n) \,\,\,\mid\,\,\,\gamma_i \in \Gamma(F), x, y \in %\overline{\Gamma}\}\rightarrow \Gamma^n
%\end{equation}
%such that $ f(x, y) := \tr \circ \o \circ \psi$ is a non-constant polynomials on two wariables $x, y$. Hence there is an element $\alpha \in F$ such that
%$f(x, \alpha)$ is non-constant polynomial on one variables. Then the restriction $\phi$ of $\psi$ on the subsets of matrices in \ref{equa1001} with $y %=\alpha$ give us an appropriate morphism.
%\end{proof}
By Lemma \ref{lem972}, we have a non-constant polynomial
$\Phi(x, y)\in \Q[x, y]$ such that $\Phi(0,0) = 2$ and $\Phi(\alpha, \beta) \in \Imm\, \tr\circ\w$ for every  $\alpha, \beta \in \Q$. Then we can  find a rational number $\beta$ such that $f(x)\coloneqq \Phi(x, \beta)$ is a non-constant polynomial and $f(\alpha) \in \Imm\, \tr\circ\w$ for every $\alpha \in \Q$.

Put
$$\X_L \coloneqq \{r = f(q) \,\,\mid\,\,\,q\in \Q, \sqrt{ f(q)^2 - 4}\in L\}.$$

\begin{lemma}
\label{lem573}
Suppose that $\X_L$ is an infinite set. Then the statement of Theorem \ref{th597} holds.
\end{lemma}

\begin{proof}
Let $q\in \Q$. Then $f(q) = \tr g$ for some element $g\in \Imm\,w$. We may assume $f(q) \ne \pm 2$.
Then $g$ is a split semisimple element in $\SL_2(L)$ if and only if
$\sqrt{\tr(g)^2 - 4}\in L$.
Moreover, if $\tr (g_1) \ne  \tr (g_2)$ for $g_1, g_2 \in \SL_2(L)$, then $g_1, g_2$ are in different conjugacy classes of $\SL_2(L)$.
Thus, if the set $\X_L$ is infinite, there are infinitely many elements of $\Imm\, w$ which are split semisimple elements belonging to different conjugacy classes of $\SL_2(L)$.
\end{proof}

Obviously, we may assume that $L$ itself is a quadratically meagre field. Let $S$ be a finite set of primes containing $p_\infty$.

\begin{lemma}
\label{lem575}
There exists an infinite set $\V\subset \Q$ such that for every $p \in S$ and every $q\in \V$ we have
$$\sqrt{f(q)^2 - 4}\in \Q_p.$$
\end{lemma}

\begin{proof}
Let $\Psi (x, y, z) = cx^r - cy^s + z\phi (y, z)\in \Q[x, y, z]$ where $\bf c \ne 0$. For every prime $p \in S$ the equation $\Psi (x,y, z) = 0$
defines a surface $X_{\Q_p}$  in the affine space ${\mathbb A}_{\Q_p}^3$. Since for $a = (1, 1, 0)$ we have $\Psi (a) = 0$ and $(\frac{\partial\Psi}{\partial x})_{a}\ne 0$,
by the implicit function theorem
there exist a neighbourhood of $a$ in
${\mathbb A}^3_{\Q_p}$
$$U_{p,a} = U_{p,1}^x \times U_{p, 1}^y\times  U_{p, 0}^z$$
where
$$U_{p,1}^x = \{\alpha\in \Q_p \,\mid\,\,\|\alpha -1\|_p < \varepsilon\},$$
$$U_{p,1}^y = \{\beta\in \Q_p\,\,\mid\,\,\|\beta -1\|_p < \varepsilon\},$$
$$U_{p,0}^z = \{\gamma\in \Q_p,\,\mid\,\|\gamma\|_p < \varepsilon\},$$
$\varepsilon \in  \R_{>0}$, and a smooth continuous function with respect to the topology induced by the natural topology on $\Q_p$
$$\theta_p\colon  U_{p,1}^y\times  U_{p, 0}^z\rightarrow U_{p,1}^x$$
such that

\begin{equation}
\label{equa11}
(\theta_p((\beta, \gamma)), \beta, \gamma) \in X_{\Q_p}\,\,\,\text{for every}\,\,\, \beta \in U_{p,1}^y,\,\gamma \in U_{p, 0}^z.
\end{equation}

Put
$$U^y_{S, 1} = \prod_{p\in S} U_{p,1}^y,\,\, U^z_{S, 0} = \prod_{p\in S} U_{p,0}^z.$$
The sets $U^y_{S, 1}, U^z_{S, 0}$ are  neighbourhoods of $1$  and $0$ in $\prod_{p\in S}\Q_p$, respectively.
Since the subset $\Q\subset \prod_{p\in S}\Q_p$ is dense in $  \prod_{p\in S}\Q_p$ by the weak approximation theorem,
the sets $\Q\cap U^y_{S, 1}, \Q\cap U^z_{S, 0}$ are infinite. Moreover, the set
$$\V \coloneqq \{q \in \Q\,\,\,\mid\,\,\, q = \frac{q_y}{q_z},q_y\in \Q\cap U^y_{S, 1}, 0\ne  q_z\in\Q\cap U^z_{S, 0}\}$$
is infinite (indeed, for every $p\in S$ the value $ \|q_y\|_p$ is bounded and the value $\|q_z\|_p$ can be made smaller than
any positive $\varepsilon \in \R$).

Now let $ f (t) = c_0t^d + c_1t^{d-1} + \cdots + c_d$ (here we change the variable $x$ to $t$). Put $t = y/z$. Then
$$  \frac{c_0^2 x^{2d}}{z^{2d}} -  (f(y/z)^2 - 4)  = \frac{c_0^2 x^{2d} - c_0^2y^{2d} +  z\phi(y, z)}{z^{2d}}$$
for some $\phi(y, z)\in \Q[y, z]$. Take $\Psi (x, y, z) = c_0^2 x^{2d} - c_0^2y^{2d} +  z\phi(y, z).$

For every $p \in S$ we obtain from (\ref{equa11})  that for every $q_y\in \Q\cap U^y_{S, 1}$, $q_z \in \Q^*\cap U^z_{S, 0}$
we have
\begin{equation}
\label{equa17}
\underbrace{\frac{c_0^2 q_y^{2d}  - q_z \phi(q_y, q_z)}{q_z^{2d}}}_{f(q_y/q_z)^2 -4 } = \frac{c_0^2\theta_p(q_y, q_z)^{2d}}{q_z^{2d}}\in \Q_p^{*2}.
\end{equation}

Thus from (\ref{equa17}) and the definition of $\V$ we obtain the statement of the lemma.
\end{proof}

Now we can prove Theorem \ref{th597}. Put $S: = S_L$ and
$$\X_L^\prime  = \{r =  f (q) \,\,\mid\,\, q \in \V\}.$$
Then $\X_L^\prime$ is an infinite  set of positive rational numbers (by Lemma \ref{lem575}).

\begin{lemma}
\label{lem576}
$\X_L^\prime \subset \X_L.$
\end{lemma}

\begin{proof}
Let $r\in X_L^\prime$, and let  $s:= r^2 -4 = c/d $ with $(c, d)=1$. Since $p_\infty \in S_L = S$, we have $\sqrt{s}\in Q_{p_\infty} = \R$, and therefore $s >  0$.
Denote  $$\bar{s} \coloneqq \text{the squarefree part of the integer}\,\,\, sd^2 = cd.$$  Then
$\Q(\sqrt{s}) = \Q(\sqrt{\bar{s}})$
and
$\Q_p(\sqrt{s}) = \Q_p(\sqrt{\bar{s}})$
for every $p \in S$. Since $s$ is a square in $\Q_p$ (Lemma \ref{lem575}), we have no $p$ from $S$
in the decomposition $\bar{s}= p_1p_2\cdots p_r$. Hence $\sqrt{\bar{s}}\in L$  according to the definition of $S = S_L$. Thus we have the inclusion $\X_L^\prime \subset \X_L.$
\end{proof}
Now the statement of the theorem follows from Lemmas \ref{lem573}, \ref{lem575}, \ref{lem576}.
\end{proof}

\begin{remark}
Probably, with appropriate changes Theorem \ref{th597} can be extended to the case $\ch \,L = p >0$.
%{\bf UBRAT FRAZU!} Anyway, the statement of the theorem holds for all fields
%containing a maximal $2$-extension $F$ of a prime field $F_p$.
%Indeed, for every finite extension $F_p\subset L\subset F$ the semisimple element  %$g \in \w(\SL_2(L)^n$ splits in $\SL_2(F) \supset \w(\SL_2(F)^n)\supset %\w(\SL_2(L)^n).$
%EXPLAIN BETTER!
\end{remark}

%The following simple observation could eventually be useful for further investigation of word maps.
%{\bf Ia bi ubral sleduuscsee Proposition: because the number of conjugacy classes of elemenets of the order < m is finite}
%\begin{prop}
%\label{pr711}
%Let $w\in F_d$ be a non-trivial word, and let $m$ be a natural number. Then for any infinite field $L$
%there exists an extension $F/L$, $[F:L]\leq 2$, such that
%the image of the word map $\w\colon (\SL_2(F))^d\rightarrow\SL_2(F)$ contains a split semisimple element of order $\geq m$.
%\end{prop}

%\begin{proof}
%In the image of the restricted map $\Res_{\SL_2(L)^n}\w$ there are infinitely many elements with different traces (see \cite{HLS}),
%and therefore there is a semisimple element $g$ of order
%$\geq m$. If $g$ is not split, it will be split in a quadratic extension $F$ of $L$.
%\end{proof}
%{\bf Eto tozhe ubrat nado?}
%\begin{remark}
%Let $L = \Q$. Proposition \ref{pr711} implies that for every non-trivial $w\in F_d$ and  $m\in \N$
%one can find an extension $F = \Q(\sqrt{a})$ for some $a \in \Z$ such that
%the set $\w(\SL_2(F)^d)$ contains a split semisimple element of $\SL_2(F)$. Moreover,
%one can see that the modules of the values $\{\left\vert f(q)\right\vert \,\,q \in \Q\}$ of the polynomial $f  = \tr \circ \w\circ \in \Q[t]$ %from Lemma \ref{lem397}
%are not bounded, and therefore one can find $q\in \Q$ such that $a = f^2(q) - 4 >0$. Thus we may take the extension $F/\Q$ to be real.
%\end{remark}

\subsection{Search for unipotent elements in groups of Lie rank 1}
Surprisingly enough, the situation here is much more complicated even in the
case $G=\SL_2(\C)$ (see Question \ref{q:sur}(i)). In fact, since all unipotent
elements of $G$ are conjugate, to guarantee that
all unipotent elements belong to the image of the word map $\w\colon G^d\rightarrow G$, it is enough to
prove this for a {\it single} element $u=\left(\begin{matrix} 1 & 1\\ 0 & 1
\end{matrix}\right)$. However, as for now, this is known only for certain families
of word maps. The main approaches used so far are based on
\begin{itemize}
\item[(i)] the Magnus embedding (see \cite{BaZa});
\item[(ii)] the representation varieties of the one-relator groups $F_2/\left<w\right>$
(see \cite{GKP1}--\cite{GKP2}).
\end{itemize}
We do not present any details here referring the reader to the papers cited above
and limiting ourselves to sketching the main ideas.

The first approach relies on the following (clever modification of the)
construction of Magnus (see \cite{Mag} and \cite{We}). First,
to each generator $x_i$ of $F_d$
one can associate an upper-triangular matrix with determinant one
$$
\left(\begin{matrix} t_i & s_i \\ 0 & t_i^{-1}\end{matrix}\right)
$$
over the ring $R_d=\Z[t_1,t_1^{-1},\dots ,t_d,t_d^{-1}, s_1,\dots ,s_d]$.
In the papers cited above it is shown that this correspondence extends to an embedding
$F_d/F_d^{(2)}$ into the group $B(R_d)$ of unimodular upper-triangular matrices over $R_d$
(here $F_d^{(2)}$ denotes
the second derived subgroup of $F_d$). For a field $K$ of transcendence degree at least $2d$ over $\Q$
this gives an embedding of $F_d/F_d^{(2)}$ into $B(K)$, the group of unimodular upper-triangular matrices over $K$. 
Let $w\in F_d^{(1)}\setminus F_d^{(2)}$. Then the word map $\w\colon B(\C)^d\rightarrow U(\C)\coloneqq [B(\C), B(C)]$ is surjective. 
For every $L\leq \C$ the subgroup $B(L)$ is dense in $B(\C)$. Hence there is a non-trivial element in $\w(B(L))$. Thus, we 
have a unipotent element in $\w(\SL_2(L))$. The existence of a unipotent element for words $w\in F_d \setminus F_1$ is obvious 
(it is enough to restrict $\w$ to $U(L)$). Hence we have the following theorem, due to Bandman and Zarhin. 

\begin{theorem} \cite{BaZa} 
Let  $L$  be a field of characteristic zero, and let $w \in F_n\setminus F_n^2$.
Further, let $G = \SL_2(L)$, and let $\w\colon G^n\rightarrow G$ be the corresponding word map.
Then the set $\Imm\, \w$ contains a non-trivial unipotent element.
\end{theorem}

%{\bf THE NEXT CHAPTER IS NOT CLEAR. WHAT WE DO HERE? MAKE LOWER DEGREE OF TRANSCENDENCE? WHAT FOR? Eto verno dla polya char 0!!! Ia bi ubral 2 abzaca nize!!
%Further, this construction can be improved so that to prove that the assignment
%$$
%g_i\mapsto \left(\begin{bmatrix} t_i & 0 \\ 1 & t_i^{-1}\end{bmatrix}\right)
%$$
%extends to a group homomorphism $\mu\colon F_d\to ST(\Lambda_d$ with kernel $F^{(2)}$
%(where $\Lambda_d=\Z[t_1,t_1^{-1},\dots ,t_d,t_d^{-1}]$), hence induces
%an embedding of $F_d/F_d^{(2)}$ into $ST(\Lambda_n)$, and therefore
%into $ST(2,K)$ where $K$ is of transcendence degree at least $d$ over $\Q$.

%Finally, one can show that the kernel of $\mu$ is $F_d^{(2)}$ and
%for $w\in F_d^{(1)}\setminus F_d^{(2)}$ there exists a nonzero Laurent series $L_w\in \Lambda_d$
%such that $\mu (w)=\left(\begin{matrix} 1 & 0 \\ L_w & 1\end{matrix}\right)$.
%By specializing the $t_i$,
%we arrive at a non-trivial unipotent in $G$. For words $w\notin F_d^{(1)}$ the surjectivity
%is obvious since the problem reduces to a power map. Evidently, this construction does not
%work for $w\in F_d^{(2)}$.}
%{\bf Abzac vishe nado bi podrobnee. In particular, at the beginning we suppose that field has enough transcendent elements . Then we use that for %any  infinite field $L$ the group $B(L)$ is dense in $B$?}

%\medskip

The second approach is based on the study of the structure of the representation variety
$$R (\Gamma_w, \SL_2(\C)) =\{ \rho\colon \Gamma_w \rightarrow \SL_2(\C)\}.$$
Namely, it can be identified with
$$\W_w =  \{(g_1, \dots, g_d) \in G^d\,\,\,\mid\,\,\,  \w(g_1, \dots, g_d) = 1\}$$
(see \cite[page~4]{LM}) and thus embeds into
$$\T_w = \{(g_1, \dots, g_d) \in G^d\,\,\,\mid\,\,\, \tr \w(g_1, \dots, g_d) = 2\}.$$
Thus, $\T_w$ is the variety of all elements $\gamma$ in $G^d$ such that $\w(\gamma)$ is a unipotent element of $G$. Obviously, $\W_w \subseteq  \T_w$. Then
\begin{equation}
\label{equa1111}
\W_w \ne  \T_w \Rightarrow \text{all unipotent elements belong to}\,\,\,\Imm\,\w.
\end{equation}
Looking at the irreducible components of these varieties, one can notice that all
components of $\T_w$ are of dimension $3d-1$. Hence, once $\W_w$ has a component
of smaller dimension, one can deduce that it is properly included in a component
of $\T_w$, so that the image of $\w$ contains all unipotent elements of $G$.

This method requires heavy computations, so that longer is $w$, sooner we arrive
at the limit of computer resources, even if for detecting small-dimensional components we replace $\W_w$ with the character variety $\W_w\sslash G$.

\begin{remark}
We do not know whether the implication converse to \eqref{equa1111}:
$$\text{all unipotent elements belong to}\,\,\,\Imm\,\w \Rightarrow  \W_w \ne  \T_w$$
is true.
\end{remark}

%\bigskip

To get semisimple and unipotent elements in the image of a word map on groups of higher Lie rank,
one can use the following classical construction.

\subsection{Embedding of $\SL_2(L)$ into simple groups.} \label{sec:emb} Let $L$ be a field, let $\G$ be a simple
%simply connected {\bf We need anywhere simply-conectnes?},
linear algebraic group defined and split over $L$, and let $G = \G(L)$.
The existence of appropriate homomorphisms  $\xi\colon \SL_2(L)\rightarrow G$ gives us a tool for investigating word maps, in particular,
for reducing some questions on word maps on $G$ to the corresponding questions for $\SL_2(L)$.

\bigskip

I. {\it The Morozov--Jacobson embedding.}  Let $L$ be a field of characteristic zero, and
let $u \in G=\G(L)$ be a unipotent element. Then there is a closed subgroup  $\tilde{\Gamma} \leq \G$ such that the subgroup $\Gamma := \tilde{\Gamma}(L) \leq G$ contains $u$ and is isomorphic either to $\SL_2(L)$ or to $\PSL_2(L)$,
see, e.g., \cite[7.4, 10.2]{Hu}; it is not so easy to distinguish between the two groups of rank 1 mentioned above,
see the discussion in \cite{MO3}.

   %(If there is a not very old reference?).}
Further,  let $\w\colon G^d\rightarrow G$ be a word map, and let
$\Res_\Gamma \, \w\colon \Gamma^d \rightarrow \Gamma $ be its restriction to $\Gamma$.
Let $\xi\colon \SL_2(L) \rightarrow G$ be a homomorphism such that $\Imm\,\xi = \Gamma$.  Denote by $\w^\prime \colon \SL_2 (L)^d\rightarrow\SL_2(L)$ the word map
induced by the same word $w\in F_d$.

I.1. Suppose that there exists a non-trivial unipotent element $u^\prime \in \Imm\,\w^\prime$.  Then $$\xi(u^\prime) \in \Imm\,(\Res_\Gamma\,\w)  \subset \Imm\,\w.$$
In particular, one can get any unipotent element in $\Imm\,\w$ (this was noticed in \cite{BaZa}).

I.2. Let $U$ be a maximal  unipotent subgroup of $G$ normalized by the group $T = \T(L)$ where $\T$ is a maximal split torus of $\G$,
let $u\in U$ be a regular unipotent element of $G$, and let $\Gamma \leq G$, $\Gamma \approx \SL_2(L)$ or $\PSL_2(L)$,
be a subgroup containing $u$.
Let $T_\Gamma \leq \Gamma$ be a maximal torus of $\Gamma$. We may assume $T_\Gamma \leq T$.

The following fact is well known, however we give a proof being unable to provide a reference.

%{\it $\flat:$

\begin{prop} \label{flat}
If the order of $\,\,t \in T_\Gamma$ is large enough, then $\xi (t)$  is a regular semisimple element of $G$.
\end{prop}

\begin{proof}
%{\bf !!!! Simply connected not necessary : we present unipotent element}{\bf Ubrat !!!Since $\G$ is simply connected, the group $G = \G(L)$}
%\bigskip
Let $R$ be the root system corresponding to $\G$. Fix
$\Pi = \{\alpha_1, \dots, \alpha_r\}$, a collection of roots corresponding to $\T$, then the group $U$ is generated by the root subgroups $\langle x_\alpha(t)\,\,\,\mid\,\,\,t \in L, \alpha \in \R^+\rangle$
(see, e.g., Section 5 of \cite{St2}), and any regular unipotent $u\in U$ is of the form
$$u = x_{\alpha_1}(a_1) x_{\alpha_2}(a_2) \cdots x_{\alpha_r}(a_r)u^*$$
where $a_i \ne 0, a_i \in L$ for every $i$, $u^* \in [U, U]$ (see \cite[3.1.13]{SS}, \cite[Lemma~3.2c]{St1}).
Then for every $t \in T_\Gamma$ we have
$$tut^{-1} =  x_{\alpha_1}(\chi_1(t)s_1) x_{\alpha_2}(\chi_2(t)s_2) \cdots x_{\alpha_r}(\chi_r(t)s_r)u^{**}$$
where $s_i \ne 0$ for every $i$, $u^{**} \in [U, U]$, and
$\chi_i \colon T_\Gamma\rightarrow L^*$ is the character of $T_\Gamma$ which corresponds to the root $\alpha$. %{\bf BEGIN CHANGE}
%On the other hand, if $t \in T_\Gamma$ then $t$ normalizes the group $\overline{\left< u\right>}$ which is the group of upper triangle unipotent mtrices of $\Gamma = \SL_2(L)\leq G$ or $\Gamma =\PSL_2(L)\leq G$.
Let $u^\prime$ be a unipotent element of $\SL_2(L)$ such that $\xi(u^\prime) = u$, and let $t^\prime\in \SL_2(L)$ be an element such that
$\xi(t^\prime) = t\in T_\Gamma$.
We can also identify $ u^\prime$ with a matrix $\begin{pmatrix} 1& a\cr 0&1\end{pmatrix}$ for some $a\in L^*$
and $t^\prime$ with a matrix  of the form $\begin{pmatrix} s&0\cr 0&s^{-1}\end{pmatrix}$. Since $\ch L = 0$,  we have an infinite set of powers  $u^{\prime n}$ among elements of the form $t^\prime u^\prime t^{\prime -1}$ for some $t^\prime \in \SL_2(L)$. Then the set
$\{t u t^{-1}\,\,\mid\,\,t \in T_\Gamma\}$ contains infinitely many elements of the form $u^m, m\in \Z$. This implies, in its turn, that
$\chi_1(t) = \chi_2 (t)= \cdots =\chi_r(t)$ for infinitely many elements $t \in T_\Gamma$. Further, all characters $\chi_i\colon T_\Gamma\rightarrow L^*$
are obtained by restricting characters of the torus $\T$ to the one-dimensional subtorus $\T_{\tilde{\Gamma}}\coloneqq\tilde{\Gamma}\cap \T$ and then on its $L$-points
$T_\Gamma = \T_{\tilde{\Gamma}}(L)$. Since the characters of any torus are continuous with respect to Zariski topology, the coincidence of characters of the one-dimensional torus $\T_{\tilde{\Gamma}}$ on an infinite set implies that these are the same characters, and therefore all restrictions
$\chi_i \colon T_\Gamma\rightarrow L^*$ are equal to a character  $\chi\colon  T_\Gamma\rightarrow L^*$.
Since every positive root $\alpha$ is a sum of the roots $\alpha_i$,  the corresponding character $\chi_\alpha\colon T_\Gamma\rightarrow L^*$
defined by the formula $tx_\alpha(s)t^{-1} = x_\alpha (\chi_\alpha (t)s)$ is equal to $\chi^N$ for some $N >0$. Then, if $t\in T_\Gamma$ is an element of sufficiently
large order, $tx_\alpha (s) t^{-1} \ne x_\alpha (s)$ for every $\alpha\in R^+$, and therefore $t$ is a regular element. %{\bf END CHANGES}
\end{proof}

The following fact, used in \cite{HLS}, is an immediate consequence of Proposition \ref{flat}.

\begin{prop} \label{sharp}
If $t \in \Imm\,w^\prime$ is a split semisimple element of sufficiently large order, then $\xi(t) \in \Imm\, \w$ is a split regular semisimple element of $G$.
\end{prop}

II. {\it The Testerman embedding}.  Let $\ch \, L = p > 0$.
Then the previous constructions from the characteristic zero case have the following constraint:
one can put a unipotent element  $u \in G$ in the image of a homomorphism $\xi\colon \SL_2(L)\rightarrow G$
only if the order of $u$ is equal to $p$. It turns out that this condition
on the order of $u$ is sufficient. The following theorem  was proved in \cite{Te}
for ``good'' primes
(see also \cite{McN} for a streamlined proof). The case of ``bad''
primes was treated in \cite{PST}.

\begin{theorem} \label{Test}
Let $G$ be a simple algebraic group over an algebraically closed field of characteristic $p>0$. Let $u\in G$ be a unipotent element. Then $u$ is contained
in a closed connected subgroup $\Gamma\le G$ of type $\sA_1$, except for the case $p=3$, $G=\sG_2$,
$u$ is an element of order 3 lying in a certain conjugacy class (labelled $\sA^{(3)}_1$).
\end{theorem}
%\bigskip

%II.1.

Here is an immediate consequence.

\begin{corollary}
Let $p$, $G$, $\Gamma $ and $u$ be as in Theorem \ref{Test}. Let $w\in F_d$ be a non-trivial word, and
let $\w^\prime\colon \Gamma^d\to \Gamma$ be the corresponding word map. Suppose that there exists a non-trivial unipotent element $u^\prime \in \Imm\,\w^\prime$. Then $u$ belongs to the image of
$\w\colon G^d\to G$.
\end{corollary}

\section{Problems of Waring type} \label{sec:Waring}

%There are some results about the images of word maps $\w: G^n\rightarrow G$ in the cases when $\G$ is a simple algebraic group and $G = \G(K)$ in the cases when $K$ is not an algebraically closed field. These results deal %with two polar cases: when $\G$ is split ( quasi-split) and when $\G$ is anisotropic group (the letter case is mostly considered for $K = \R$).
%Most of results may be called negative-positive where the negative results are obtained by fixing a group and changing  words and positive results, respectively, are obtained by fixing a word and "increasing " a group.

If $\G$ is a semisimple algebraic group over an algebraically closed field $K$, $G=\G(K)$ and $w\in F_d$ is a non-trivial word,
then, even if the surjectivity of the word map $\w\colon G^d\to G$ is unknown (or is known to fail), the Borel dominance theorem
guarantees that every element $g\in G$ can be represented as a product of at most two $w$-values: $g=g_1g_2$
with $g_i \in \Imm \w$. However, Thom's phenomenon discussed in Section \ref{Thom} shows that this is not necessarily the case
when the base field is not algebraically closed. Moreover, the proof of Thom's theorem shows that if $w$ varies,
the $w$-width of a compact real group $G=\cG(\R)$ can be made as large as we wish.

Indeed, fix $\varepsilon >0$. Let $w$ be a Thom word, i.e., the image of $\w$ is contained in the $\varepsilon$-neighbourhood of the identity element of $G$.
Then given a positive integer $k$, one can easily prove (say, by induction on $k$)
that for any $g_1,\dots, g_k\in \Imm\,w$ we have (with the notation of Theorem \ref{thom-hls}(i)) $l(g_1\cdots g_k)=d(1, g_1g_2\cdots g_k) <
k\varepsilon $. Hence, taking
smaller $\varepsilon$ and choosing an appropriate Thom's word $w$, one can make the $w$-width of $G$
larger than any given positive integer.

\medskip

However, for {\it split} groups the situation is not that hopeless, though also here one can observe negative-positive results.

\medskip

Let $\G$ be a split, simple, simply connected linear algebraic group defined over a field $K$ (not necessarily algebraically closed).
Then the group $G= \G(K)$ is a {\it quasi-simple} abstract group (that is, $G=[G,G]$ and $G/Z(G)$ is simple), except for $G = \SL_2(\mathbb F_2)$, $\SL_2(\mathbb F_3)$,   $\SU_3(\mathbb F_4)$, $\sB_2(2)$, $\sG_2(2)$.

There are two different cases to be considered separately: finite and infinite ground fields.

\subsection{Case of finite fields}
In the case of a finite ground field $K$, Borel's dominance theorem is even less meaningful
than in the case where the ground field is real or $p$-adic. So one can consider
the $w$-width as a reasonable measure for the size of the image of the word map $\w$. The aim is to obtain results of the flavour of Theorem \ref{thom-hls}(ii), which guarantee that
every element of $G$ can be represented as a product of ``small'' number of $w$-values.

Recall that if $K=\mathbb F_q$, apart from the case where $G$ is split giving rise to the
abstract simple groups $\G(q)$ (Chevalley groups), we have several additional
series.  Namely, let $\G$ be a connected, simple, simply connected linear algebraic $K$-group. Since $K$ is finite, by a theorem
of Lang $\G$ is quasi-split (that is, has a $K$-defined Borel subgroup). If $\G$ is
not split, we have {\it twisted forms} of Chevalley groups (sometimes called
Steinberg groups) of types $^2\sA_r$ $(r>1)$, $^2\sD_r$ $(r>3)$, $^3\sD_4$, $^2\sE_6$,
whose groups of $K$-points $\G(K)$ are quasi-simple abstract groups. If we add the abstract groups of types
$^2\sB_2(2^{2m+1})$, $^2\sG_2(3^{2m+1})$, $^2\sF_4(2^{2m+1})$ (the Suzuki and Ree groups), each of which is obtained as the group of fixed points of an appropriate automorphism of the corresponding simple algebraic group, exclude
$^2\sB_2(2)$ and $^2\sG_2(3)$ that are not quasi-simple, and replace $^2\sF_4(2)$
with its derived subgroup (called the Tits group),
we obtain the main infinite family of finite non-abelian simple groups called finite simple groups of Lie type. Together with
the family of alternating groups $A_n$ and 26  sporadic groups, these are all finite simple groups. Thus, any general result on word maps on finite simple groups can also be viewed as a result on word maps on groups of points of a simple  split (or quasi-split) algebraic group over a finite field (up to the centre).

Here we have the following negative-positive result:

\begin{theorem} \label{th:F}
${}$
\begin{itemize}
\item[(i)]
Let $G$ be a finite non-abelian simple group, and let $A$ be an $\operatorname{Aut}(G)$-invariant subset of $G$ such that $1 \in A$.
Then there exists a word $w \in  F_2$ such that $\Imm \w = A$.
\item[(ii)]
Let $1\ne w_1(x_1, \dots, x_n)\in F_n, \,\,1\ne w_2(y_1, \dots, y_m)\in F_m, \,\,\,w = w_1 w_2$.
Then there exists $c = c(w_1,w_2)$ such that for every quasi-simple group $G$ of order greater than $c$ the image of $\w\colon G^{n+m}\rightarrow G$ contains $G\setminus Z(G)$.
\end{itemize}
\end{theorem}

Statement (i) is a theorem of A.~Lubotzky  \cite{Lu}, showing that the image of a word map
can be made as small as possible, within the inevitable natural constraints (the image must contain 1
and be invariant under any automorphism), if one fixes $G$ and varies $w$. (Earlier results of this flavour were obtained
by M.~Kassabov and N.~Nikolov \cite{KN}, and M.~Levy \cite{Levy1} for some families of finite simple groups.)

The proof of (i) is based on the ``one-and-a-half'' generation theorem \cite{GK}, \cite{Stein}: for every element $a\ne 1$ of a finite non-abelian simple group $G$ there exists $b\in G$ such that $\langle a, b\rangle = G$. The proof is tricky enough and gives the following interesting result: if $G$ is a
 finite non-abelian simple group, then there is a word $w\in F_2$ such that for every
 $(a, b) \in G\times G$ we have
 $$w(a, b) \ne 1\Leftrightarrow \langle a, b\rangle = G.$$
 Since we may view the word $w\in F_2$ as an element in $F_d$, $d>2$, we may formulate the result also for words in $F_d$.

\begin{remark}
This negative result shows that in the positive result of (ii) one cannot drop the assumption that
the rank of $G$ is large enough. Indeed, in the situation of (i), one can choose $A$ to be a single conjugacy class so that
for the pair $(w,G)$ the $w$-width of $G$ will be greater than 2.
\end{remark}

\begin{remark}
Statement (i) was extended by M.~Levy to quasi-simple and almost simple finite groups \cite{Levy2}.
\end{remark}

Statement (ii) (which should be compared with Theorem \ref{thom-hls}(ii))
is a theorem of Guralnick and Tiep \cite{GT}, who made a final step along the road paved
in two earlier papers of Larsen--Shalev--Tiep \cite{LST1}, \cite{LST2}. The proof is difficult.
The principal part, contained in \cite{LST1}, is mainly based on the Deligne--Lusztig theory of characters combined with some
arithmetic-geometric properties of groups of Lie type. The latter ones include a delicate theorem of Chebotarev
flavour which guarantees the existence of regular semisimple elements in the image
of $\w$ lying in a split maximal torus of $G$ and is proven
with the help of high-tech machinery (Lefschetz' trace formula and estimates of Lang--Weil type).
Using these methods, the authors finally prove that for a given pair of words  $w_1,w_2$ and a big group $G$
there are special semisimple conjugacy classes $C_1, C_2$ such that $C_1C_2 \supseteq G\setminus \{1\}$ and $C_1 \subset \Imm\, \w_1, C_2 \subset \Imm\, \w_2$.
Since $1$ is contained in the image of every word map, we have $\Imm\, \w_1 \Imm\, \w_2 = G$, and therefore the map $\w$ is surjective.

In \cite{LST2}, results and constructions from \cite{LST1} are extended to the case
where $G$ is {\it quasi-simple}, so that to get the word width at most 3, with
exhibiting central elements of word length 3 obstructing to improve that to 2, but
leaving open the question whether all {\it non-central} elements are of length
at most 2. This last step was done in \cite{GT}, with significant effort, using
subtle group-theoretic arguments (such as looking for regular elements of special form) combined with some facts from spinor theory.

\begin{remark}
We do not know if statement (ii) can be extended to the cases where $w$ is a product of two {\it non-disjoint} words $w_1w_2$.
A natural constraint here is that the word $w$ must not be representable
as a proper power of another word: $w \ne w_1^k$ for $k >1$.
\end{remark}

\begin{remark}
Theorem \ref{th:F} concerns arbitrary words $w$. The behaviour of particular
word maps on finite simple groups has been a subject of intense study over several decades.
We only present here a brief account of main achievements, often giving only final results
and omitting the preceding contributions.

\begin{itemize}
\item[(i)] {\it Commutator $w=[x,y]=xyx^{-1}y^{-1}$}. The map $\w\colon G^2\to G$ is surjective
on all finite simple groups $G$ \cite{LOST1}. If $G$ is quasi-simple, $\wid_w(G)\le 2$, the
estimate is sharp, and all groups with $\wid_w(G)=2$ are listed \cite{LOST2}. See \cite{Mall}
for a detailed survey of this longstanding problem.
\item[(ii)] {\it Words that are not surjective on infinitely many finite simple groups}.
A family of words with this property was constructed by Jambor, Liebeck and O'${}^{\prime}$Brien \cite{JLO},
the simplest of them is $w=x^2[x^{-2},y^{-1}]^2$, which is not surjective on $\PSL (2,\mathbb F_p)$
for infinitely many $p$.
\item[(iii)] {\it Power words $w=x^n$.} For obvious reasons, here one cannot expect any general surjectivity
result because the image of $\w$ collapses to 1 for all groups of order divisible by $n$. The main
problem consists in the computation of $\wid_w(G)$. Almost all results in this direction have been
superseded by the paper of Guralnick, Liebeck, O${}^{\prime}$Brien, Shalev, and Tiep \cite{GLOST}. Let us quote some
of their fundamental results.

      \begin{itemize}
\item [(1)] Let $N=p^aq^b$ where $p$, $q$ are prime numbers and $a,b$ are non-negative integers. Then the word map induced by $w(x,y)=x^Ny^N$ is surjective on all finite non-abelian simple groups.
\item [(2)] Let $N$ be an odd positive integer. Then the word map induced by $w(x,y,z)=x^Ny^Nz^N$ is surjective on all finite quasi-simple groups.
\item [(3)] Let $N=p_1^{\alpha_1}\cdots p_k^{\alpha_k}$ $(p_1<\cdots <p_k$, $\alpha_i>0)$
be the prime decomposition of $N$, let $\pi (N)\coloneqq k$, and let $\Omega (N):=\sum_{i=1}^k \alpha_i$. Suppose that $N$ runs through a set $S\subset\mathbb N$
such that either (a) $\Omega (N)$, or (b) $\pi (N)$ is bounded by some constant $C$. Then
for every $N\in S$ the word map induced by $w=x^Ny^N$ is surjective on all sufficiently large finite simple groups $G$. Here in case (a) this means that
for a certain function $f$ the degree $n$ (resp. the Lie rank) of $G$ must be greater than
$f(C)$ if $G=A_n$ (resp. $G$ is of Lie type), whereas in case (b) also the size
of $\mathbb F_q$ must be greater than $f(C)$ if $G$ is of Lie type over $\mathbb F_q$.
     \end{itemize}
Some comments are in order. First, note that (1) and (2) can be viewed as analogues
of the Burnside and Feit--Thompson theorems, respectively. Second, both (1) and (2)
hold for {\it all} finite simple groups $G$, similar to (i) and being in contrast
with Theorem \ref{th:F}(ii) and other earlier results of such flavour, valid only
for sufficiently large groups. Finally, the authors show that all these results
are sharp. First, note that one cannot extend (i) to the case where $N$ is a product of {\it three} prime powers: look at $N=60$ and $G=A_5$. Further, (1) cannot be extended to all {\it quasi-simple} groups $G$, even in
the weak sense: it is not always true that every non-central element of $G$ lies
in the image of $x^Ny^N$. An explicit example is provided by looking at elements of order 5 in $\SL_2(5)$, none of which lies in the image of $x^{20}y^{20}$.

Furthermore, it is not true that for every odd integer $N$ the word map induced by $w=x^Ny^N$ is
surjective on every non-abelian simple group $G$ (counter-examples appear among
$\SL_2(q)$ and $^2\sG_2(q)$).

Finally, they use the fact that there are infinitely many primes $p$ with
$\Omega (p^2-1)\le 21$. Then for $N\coloneqq p(p^2-1)$ one has
$\pi (N)\le \Omega (N)\le 22$ but $w=x^Ny^N$ is an identity in $\PSL_2(p)$. Thus
(3) does not hold for finite simple groups of Lie type and bounded rank.

As to proofs, they are mostly of group-theoretic nature. Perhaps the most difficult technical point consists in constructing certain elements of 2-power order
which are regular (or close to such), in the spirit of similar considerations
in \cite{GT} for elements of $p$-power order.
\end{itemize}
\end{remark}

\begin{remark}
In \cite{LaTi}, Larsen and Tiep refined \cite{LST1} by proving that given any non-trivial words $w_1$, $w_2$, for all sufficiently large finite nonabelian simple groups $G$ one can find ``thin'' subsets $C_i\subseteq \Imm \w_i$ $(i=1,2)$ so that
$C_1C_2=G$; explicitly, one can arrange the size of $C_i$ as $O\left(\sqrt{\left\vert G\right\vert \log \left\vert G\right\vert } \right)$.
\end{remark}

\begin{remark}
One has to mention another approach to measuring the size of the image of
a word map, going back to Larsen \cite{La}. It consists in obtaining lower estimates
on the size of the image of the form $\left\vert \Imm \w\right\vert>c\left\vert G\right\vert$.
See, e.g., \cite{LS1}, \cite{NP}, \cite{GKSV} for variations on this theme.
\end{remark}

\subsection{Split groups over infinite fields}
The following result is obtained in \cite{HLS}.

\begin{theorem} [\cite{HLS}]
Let $\G$ be a simple, simply connected  algebraic group defined and split over an infinite field $K$, and let $G = \G(K)$.
Then
\begin{itemize}
\item[(i)] for any four non-trivial words $w_1 \in F_k, w_2 \in F_l, w_3 \in F_m, w_4 \in F_n$ and any infinite field $K$ the map
$$\w\colon G^{k+l+m+n}\rightarrow G\setminus Z(G),$$
where $w = w_1w_2w_3w_4$, is surjective;
\item[(ii)] if $\G = \SL_n$, $n>2$, then for any three non-trivial words $w_1 \in F_k, w_2 \in F_l, w_3 \in F_m$
and any infinite field $K$ the map
$$\w\colon G^{k+l+m}\rightarrow G\setminus Z(G),$$
where $w = w_1w_2w_3$, is surjective;
\item[(iii)] if  the field of real numbers $\R$ or the field of $p$-adic numbers $\mathbb Q_p$ is contained in $K$,
then for any two  non-trivial words $w_1 \in F_k, w_2\in F_l$
the map
$$\w\colon G^{k+l}\rightarrow G\setminus Z(G),$$
where $w = w_1w_2$, is surjective.
\end{itemize}
\end{theorem}

Here we give a sketch of proof which is almost the same as in \cite{HLS}.

First of all, note that $\G(K)$ is dense in $\G$ \cite[18.3]{Bo2}. Therefore $\Imm \, w_i$ contains infinitely many regular semisimple conjugacy clases
because the set of all regular semisimple elements is an open subset in $\G$ (see \cite{SS}) and $\w_i$ is a dominant map.
If we can find {\it split } regular semisimple elements $s_1 \in M_1, s_2 \in M_2$ for some sets $M_1, M_2\subset G$ invariant under conjugation,
then their conjugacy classes $C_i$ in $G$  are also contined in $M_i$. Thus,  $M_1 M_2 \supset G\setminus Z(G)$ because
$C_1C_2\supset G\setminus Z(G)$ (see \cite{EG1}).

\bigskip
\noindent
{\it Proof of (i)}.
Let $\Gamma  = \prod_i \Gamma_i$ be a  semisimple group where each simple component $\Gamma_i$  is of type $A_{r_i}$ for some $r_i$,
and let $\Delta$ be a maximal split torus of $\Gamma$.  Assume $\Gamma$ is defined and split over $K$. Let
$\o\colon \Gamma^d\rightarrow \Gamma$ be a non-trivial word map. Since this map is dominant by Borel's theorem \cite{Bo1} and the set of regular semisimple elements is open in $\Gamma$ \cite[III.1.11]{SS}, \cite[Cor.~5.4]{St1},
we have an open subset of regular semisimple elements in $\w(\Gamma^d)$. The set $\Gamma(K)$ is dense in $\Gamma$ \cite{Bo2}, therefore
we have a regular semisimple element $s\in \o (\Gamma(K)^d)$.
Every regular semisimple conjugacy class of $\Gamma(K)$ intersects the sets $U \dot w_c$ and $\dot w_c^{-1}U$ where $U  = (R_u(B))(K)$
is the group of rational points of the unipotent radical of a Borel subgroup corresponding to $\Delta$ and $w_c = \prod_i   w_{ci}$ is a product of Coxeter elements of the components $\Gamma_i$ 
(recall that $\dot w_c $ is a preimage of $w_c$ in the normalizer of the fixed maximal torus, see the proof of Theorem \ref{thom-hls}(ii)). Actually, this follows from the existence of canonical rational form in the groups $\SL_{r_i+1}(K)$; see also \cite[Section~3.8, Theorem~4(b)]{St3}  %{\bf here should be Steinberg. Conjugacy classes or any textbook on linear algebra. }
and \cite{EG2}).
 %and $\dot w_c^{-1} t_0 = \dot w_c^{-1}$ for some fixed $t_0\in \Delta(K)$), (see, \cite{St2}, \cite{EG2}).
Thus, if $\o_1,  \o_2$ are word maps on $\Gamma^{d_i}$, then for every $t \in \Delta(K)$
the product
$\Imm\o_1\Imm \o_2$ contains an element of the form
$$\underbrace{u_1\dot w_c }_{\in \Imm\, \w_1}\underbrace{( t \dot w_c^{-1} u_2 t^{-1})}_{\in \Imm\, \w_2} = u_1 \underbrace{[ \dot w_c, t] }_{ \coloneqq t^*\in \Delta(K)} \underbrace{(t u_2 t^{-1})}_{\coloneqq u_2^\prime\in U} = u_1t^*u_2^\prime = u_1 (t^* \underbrace{u_2^\prime u_1}_{\coloneqq u\in U} ) u_1^{-1}.  $$
Hence for every $t \in \Delta(K)$ the set $\Imm\o_1\Imm \o_2$ contains an element of the form $t^*u$  where $t^* = [ \dot w_c, t] $ and $u \in U$. Since the map
$[\dot w_c, x]\colon \Delta\rightarrow \Delta$ is surjective, the set $[\dot w_c, \Delta(K)]$ is dense in $\Delta$.
Further, such a subgroup $\Gamma\leq \G$ exists for $\Delta = \T$ (see \cite {Bo1}).
Let $\o_1, \o_2$ be the restrictions of $\w_1, \w_2$ to $\Gamma^k, \Gamma^l$, then
%{\bf CHANGE!}
$\Imm \o_1\Imm \o_2$ contains elements of the form $t^*u$ where $t^*$ runs over a dense subset of %{\bf CHANGE}
$\T$.
In particular, we can find a regular in $\G$ semisimple element $t^*$. Then $t^*u$ is conjugate to $t^*$, and we find an appropriate element in  $\Imm \w_1\Imm \w_2$.
The same arguments give us a split regular semisimple element in  $\Imm \w_3\Imm \w_4$. As mentioned above, the product
of the conjugacy classes of these elements contains all non-central elements of $G$, whence the result.
\qed

\bigskip
\noindent
{\it Proof of (ii)}.
By a result of A.~Lev \cite{Lev}, the product $C_1C_2 C_3$ of any three regular conjugacy classes in $G= \SL_n(K)$, $n \geq 3$, contains
$G\setminus Z(G)$. This implies (ii) because every image $\Imm\,w_i$ contains a regular conjugacy class.
\qed

\begin{remark}
The statement of (ii) can be strengthened by removing the restrictive assumption $n>2$. 
This can be achieved by replacing Lev's theorem with Lemma 6.1 of \cite{VW}, 
which guarantees that the product $S_1S_2S_3$ of any three regular {\it similarity} classes 
in $G= \SL_n(K)$, $n \geq 2$, contains $G\setminus Z(G)$. (By definition, two matrices from $\SL_n(K)$ 
are similar if they are conjugate in $\GL_n(K)$. The image of a word map is invariant under any 
automorphism, not only inner, whence the statement.)  
\end{remark}

\begin{remark}
It would be interesting to extend Lev's theorem to products of similarity classes (as in the previous remark) 
in an arbitrary Chevalley group. This would give us three word maps in (i) instead of four.
\end{remark}

We can slightly generalize statement (iii). Namely, we have

$(iii^\prime)$
{\it Let $K$ be a quadratically meagre field of characteristic zero. Then for any two  non-trivial words $w_1 \in F_k, w_2\in F_l$
the map
$$\w\colon G^{k+l}\rightarrow G\setminus Z(G),$$
where $w = w_1w_2$, is surjective.}

\bigskip

\bigskip
\noindent
{\it Proof of $(iii^\prime)$}.
We have to prove that both $\Imm \,w_1$ and $\Imm\, w_2$ contain split regular semisimple elements of $G$. The corresponding $\SL_2(K)$-embedding  allows us
to reduce the question to the following one: to prove the existence of infinitely many split semisimple elements in each of $\w_1(\SL_2(K))$ and $\w_2(\SL_2(K))$  (see the previous section).
This is exactly the statement of Theorem \ref{th597}.
\qed

\section{Polynomial maps on matrix algebras}

Looking at equations of form \eqref{eq:alg}, one can pose questions
similar to those discussed above for equation \eqref{eq:gr}.

First consider the case where solutions are sought in the matrix
algebra $\cA=\M (n,k)$. Most general results here were recently
obtained by A.~Kanel-Belov, S.~Malev, and L.~Rowen
\cite{KBMR1}--\cite{KBMR4}, \cite{Male} (see also \cite{Sp}). We are
not going to give a detailed overview referring the reader to the
papers cited above and to a survey given in \cite{KBKP}. Let us only
note that to ask a sensible question, one has to assume that the
polynomial $P$ is not identically zero on $\cA$ and, moreover, that
it is not central (i.e., not all of its values are scalar matrices).

Under this assumption, there are essentially two different
situations: either the image of $P$ contains at least one matrix
with nonzero trace, or it consists of traceless matrices. The second
case occurs, say, when $P$ is a Lie polynomial (where the Lie
bracket is given by additive commutator, $[X,Y]=XY-YX$), and all
questions about the polynomial map
\begin{equation} \label{map:matrix}
P\colon \M (n,k)^d\to\M (n,k)
\end{equation}
can be modified to the Lie-algebraic setting. Namely, for such a $P$
and any Lie algebra $\Fg$ one can consider the induced map
\begin{equation} \label{map:Lie}
P\colon \Fg^d\to\Fg.
\end{equation}
As in the preceding sections, it is reasonable to restrict our
attention to considering {\it simple} Lie algebras.

Here is a brief account of main results on the image of maps
\eqref{map:matrix} and \eqref{map:Lie}. Throughout we assume that
$P$ is not central. In the Lie algebra case, we assume that $\Fg$ is
simple and finite-dimensional. The ground field $k$ is either $\BR$
or $\BC$.

\begin{remark} \label{rem:pol}
${}$
\begin{itemize}
\item[(i)] Regardless of the topology under consideration (Zariski, complex, or real),
there are polynomials $P$ such that the image of \eqref{map:matrix}
is not dense \cite{KBMR1}, \cite{Male}.
\item[(ii)] There are Lie polynomials $P$ such that map
\eqref{map:Lie} is not surjective \cite{BGKP}.
\item[(iii)] For any Lie polynomial $P$ which is not identically
zero on $\Fsl (2,k)$ and for any split $\Fg$, map \eqref{map:Lie} is
dominant in Zariski topology (``weak infinitesimal Borel theorem'')
\cite{BGKP}.
\end{itemize}
\end{remark}

For {\it multilinear} (associative or Lie) polynomials, no examples
such as in Remark \ref{rem:pol}(i), (ii) are known. A more
optimistic conjecture attributed to Kaplansky and L'vov asserts that
in this case the image may be either $\Fsl (n,k)$ or $\M (n,k)$; see
\cite{KBMR1}--\cite{KBMR4}, \cite{Male}, \cite{Sp}, \cite{BW},
\cite{DK} for a number of results in this direction. An analogue
of the Kaplansky--L'vov conjecture can be formulated for other classical
Lie algebras, see \cite{AEV} for some partial results. The case of
multilinear Jordan polynomials on Jordan algebras is discussed in
\cite{Grdn}, \cite{MaOl}.

Eventual gaps between the behaviour of the maps under consideration
in real and complex case and with respect to different topologies
are still poorly understood. Here are several natural questions.

\begin{question} \label{q:alg}
${}$
\begin{itemize}
\item[(i)] Does there exist $P$ such that map \eqref{map:matrix} is
surjective for $k=\BR$ and is not surjective for $k=\BC$?
\item[(ii)]  Does there exist $P$ such that the image of $P\colon\M
(n,\BC)^d\to\M (n,\BC)$ is Zariski dense but is not dense in
Euclidean complex topology?
\end{itemize}
\end{question}

Note that if in (ii) one replaces ``complex'' with ``real'',
$P(X)=X^2$ provides an example where the image is Zariski dense but
is not dense in Euclidean topology (see Example \ref{ex:sl2r} and
also \cite{Male}).

In the Lie-algebraic case, one can ask about the existence of counterparts to
Thom's phenomenon, at least in some weak sense:

\begin{question} \label{q:Thom}
Do there exist a Lie polynomial $P$ and a compact simple real Lie
algebra $\Fg$ such that the image of map \eqref{map:Lie} is not
dense in Euclidean topology?
\end{question}

Finally, in parallel to problems of Waring type for word maps on groups, one
can ask similar questions for polynomial maps on Lie algebras. Even the simplest
case of the commutator map is far from being trivial.

For an element $z$ of a Lie algebra $L$ we call its bracket length the minimal number $\ell$
such that $z$ is representable in the form $z=[x_1,y_1]+\dots +[x_\ell,y_\ell]$ with $x_i,y_i\in L$.
We call the bracket width of $L$ the supremum of bracket lengths of its elements.

Let $L$ be a {\it simple} Lie algebra over a field $k$ (or a ring $R$).

\begin{question} \label{q:br}
${}$
\begin{itemize}
\item[(i)] Can the bracket width of $L$ be infinite?
\item[(ii)] Can it be greater than one?
\end{itemize}
\end{question}
A negative answer to Question \ref{q:br}(i) is obtained by Bergman--Nahlus \cite{BN} for any
finite-dimensional simple Lie algebra $L$ over any infinite field of characteristic different
from 2 and 3: the bracket width is bounded by 2 (the proof relies on recent two-generation
theorems by Bois \cite{Boi}). (Over $\R$, a simple proof can be found in \cite{HM}; see \cite{Go4} for the
case of arbitrary classical Lie algebras.)

Question \ref{q:br}(ii) is answered in the negative in each of the following cases:
(i) $L$ is a finite-dimensional simple split Lie algebra over any sufficiently large field
(G.~Brown \cite{Br}; R. Hirschb\"uhl \cite{Hi} provided improved estimates on
the size of the ground field);
(ii) $L$ is a simple real compact Lie algebra (here there are proofs by K.-H.~Neeb \cite[Appendix~3]{HM},
D.~{\v{Z}}.~\DJ okovi\'c, T.-Y. Tam \cite[Theorem~3.4]{DjTa},
D.~Akhiezer \cite{Akh}, A.~D'Andrea and A.~Maffei \cite{DAM}, J.~Malkoun and N.~Nahlus \cite{MaNa}; in \cite{Akh} some real non-compact algebras are also treated; see also the discussion at \url{math.stackexchange.com/questions/769881}).

In view of these results, the following question looks natural.

\begin{question}
What is the bracket width of Lie algebras of Cartan type (finite-dimensional
over fields of positive characteristic and infinite-dimensional over fields of characteristic
zero)?
\end{question}

Another rich source of simple infinite-dimensional Lie algebras (algebras of vector fields on smooth affine varieties)
was discussed in \cite{BiFut}. It is a challenging question whether among these algebras one can find those with
bracket width greater than one.

\section{Miscellanea}

To conclude, we present several remarks and questions related to the
topic of the present paper. In most cases they refer to situations
which are almost totally unexplored.

\subsection{Word maps in Kac--Moody setting} In the case where a simple
algebraic group $\cG$ under consideration is defined over the field
$K=\C((t))$ of formal Laurent series with complex coefficients, naturally
leads to affine Kac--Moody groups. Various ramifications of this set-up, both
for word maps on Kac--Moody groups and polynomial maps on Kac--Moody algebras,
are surveyed in \cite{KKMP}.

\subsection{Systems of equations} It seems very problematic to go
over from maps \eqref{map:gr}, \eqref{map:matrix}, \eqref{map:Lie}
to more general ones $G^d\to G^k$, $\cA^d\to\cA^k$, $\Fg^d\to\Fg^k$
(in other words, from equations to systems of equations). Some
particular cases were treated by N.~Gordeev and U.~Rehmann
\cite{GoRe}, and by E.~Breuillard, B.~Green, R.~Guralnick and T.~Tao
\cite{BGGT}. A promising general approach was recently proposed by
K.~Bou-Rabee and M.~Larsen \cite{BRL}.

\subsection{Equidistribution problems} One can ask how the set of
solutions of \eqref{eq:gr} or \eqref{eq:alg} depends on the
right-hand side. In other words, one can study the behaviour of the
fibres of maps \eqref{map:gr}, \eqref{map:matrix}, \eqref{map:Lie}.
The authors are not aware of anything done in this direction, in
contrast to the case of finite groups where a number of equidistribution
results are available, see, e.g., \cite{AV}, \cite{BK}, \cite{Bors},
\cite{GaSh}, \cite{KuSi}, \cite{LP},  \cite{LS2}, \cite{LS3},
\cite{Na}, \cite{PS}, \cite{Pl}; in \cite{LS4} Larsen and Shalev consider equidistribution problems
for profinite and residually finite groups. Here probabilistic aspects of the theory
naturally arise. We are not going to discuss this rich topic. The interested reader can find a survey in \cite{Sh3}.

\subsection{Functional-analytical analogues}
In the border-extending spirit of Remarks \ref{rem:dif} and
\ref{rem:Cr}, one can try to investigate polynomial maps on certain
operator algebras, particularly on those for which additive
commutator is known to behave well (for example, inducing a
surjective map); see, e.g., \cite{DS}, \cite{Ng}, \cite{KNZ},
\cite{KLT}.

\subsection{Word image and anti-automorphisms} \label{sec1.1}
We start with some general (and almost obvious) remarks regarding
${\Aut}(F_d)$- and $\Aut(G)$-invariance of the image of a word map $w\colon
G^d\to G$ on an abstract group $G$.

First, evidently $\Imm \w$ is an $\Aut (G)$-invariant subset of $G$.

Second, if $w_1, w_2\in F_d$ lie in the same $\Aut (F_d)$-orbit,
then the maps
$\w_1, \w_2 \colon G^n\rightarrow G$ have the same image.

Indeed, any group homomorphism $\varphi\colon F_d\to G$ is determined by
the $d$-tuple $(g_1=\varphi (x_1),\dots , g_d=\varphi (x_d))$. Since for
any $w\in F_d$ we have $\varphi(w)=\w(g_1,\dots ,g_d)$, the image of $\w$ coincides with the
set $\{\varphi(w)\}_{\varphi\in\Hom(F_d,G)}$, whence the claim.

%\subsection{Anti-automorphisms} \label{sec1.2}
The situation becomes much less obvious as soon as we consider anti-automorphisms
instead of automorphisms. There are several ways to formalize eventual difference between the
images of corresponding word maps. Here are two possibilities.

\begin{definition}
Let $\gamma$ be an anti-automorphism of $F_d$, and let $w\in F_d$. Denote $w^{\gamma}=\gamma (w)$ and define,
for every group $G$, $\w^{\gamma}\colon G^d\to G$ to be the evaluation map, as above. We say that $w$ is
{\em $\gamma$-chiral} if there exists $G$ such that the images of $\w$ and $\w^{\gamma}$ are different.
\end{definition}

\begin{definition}
Let $G$ be a group, and let $\gamma$ be an anti-automorphism of $G$. Define, for every $w\in F_d$,
$\w_{\gamma}\colon G^d\to G$ by $\w_{\gamma}(g_1,\dots , g_d)=\gamma(w(g_1, \dots , g_d))$. We say that $G$ is
{\em $\gamma$-chiral} if there exists $w$ such that the images of $\w$ and $\w_{\gamma}$ are different.
\end{definition}

In both cases, we say that the pair $(w,G)$ is $\gamma$-chiral (otherwise, we say that it is {\it $\gamma$-achiral}).
We omit $\gamma$ in prefixes and sub(super)-scripts
whenever the anti-automorphism is fixed and this does not lead to any confusion.

Perhaps, the simplest non-trivial case where one can observe the chirality phenomenon arises when $\gamma$ acts
on any group $G$ (including $F_d$) by inverting all its elements, $\gamma (g)=g^{-1}$. In such a case, $w^{\gamma}=w_{\gamma}$
for any $G$ and any $w$.

\begin{prop} \cite{CH}
If $\gamma$ acts by inversion, there are $\gamma$-chiral pairs $(w,G)$.
\end{prop}

\begin{remark}
The simplest way to prove the proposition, demonstrated in \cite{CH},
is to combine a theorem of Lubotzky \cite{Lu} (see Theorem \ref{th:F}(i))
with the fact that there are finite simple groups all of whose automorphisms are inner
which contain an element $g$
not conjugate to its inverse. The resulting pair $(w,G)$ is then a chiral pair because the image of
$\w$ which coincides with the conjugacy class of such an element $g$ cannot contain $g^{-1}$, which
is in the image of $\w_{\gamma}$.

However, it is not easy to give an {\it explicit} example of a chiral pair: say, for the Mathieu group $G=M_{11}$ and $g\in G$ an element of order
11, one can expect $w$ of length about $1.7\cdot 10^{244552995}$, see \cite{MO2}.

Here is another way to formalize asymmetry phenomena of this flavour, which is inspired by the \url{mathoverflow} discussion cited above.
For any word map $\w\colon G^d\to G$ and any $a\in G$ we denote by $\w_a=\{(g_1,\dots ,g_n)\,\, \mid \,\, w (g_1,\dots ,g_d)=a\}$ the fibre of
$\w$ at $a$.
We restrict our attention to considering anti-automorphisms of finite groups.
\end{remark}

\begin{definition}
Let $G$ be a finite group equipped with an anti-automor\-phism $\gamma$. We say that $G$ is {\em weakly $\gamma$-chiral} if
there exist $g\in G$ and $w\in F_d$ such that the fibres $\w_g$ and $(\w_{\gamma})_g$ are of distinct cardinalities.
In such a case, we say that $(w,G)$ is a weakly $\gamma$-chiral pair.
\end{definition}

Clearly, every $\gamma$-chiral finite group is weakly $\gamma$-chiral.
It turns out that to detect weak chirality, much shorter words $w$ can be used that can be exhibited explicitly.

\begin{example}
(N. Elkies \cite{MO2})

For $a\in G=M_{11}$ an element of order 11 and $w=x^4y^2xy^3$ the fibres $\w_a$ and $\w_{a^{-1}}$ are of cardinalities 7491 and 7458,
respectively. So $(w,G)$ is a weakly $\gamma$-chiral pair where $\gamma$ stands for the inversion map.
\end{example}

\begin{question}
Does there exist a finite group $G$ equipped with an anti-automorphism $\gamma$
which is $\gamma$-achiral but weakly $\gamma$-chiral?
\end{question}

\begin{remark}
In a somewhat similar spirit, R.~Guralnick and P.~Shumyatsky \cite{GuSh}
considered words $w$ for which the equations $w(x_1,\dots ,x_d)=g$
and $w(x_1,\dots ,x_d)=g^e$ are equivalent for all $e$ (or all $e$ prime to
the order of $G$), in the sense of the existence of a solution or the number of
solutions.
Not too much is known about the invariance of $\Imm \w$ with respect to other operations
on $F_d$ and $G$.
It would be interesting to divide words into equivalence classes with respect to certain invariance properties of  $\Imm \w$ for a given group $G$.
\end{remark}

\subsection{Word maps with constants} \label{sec:const}
One of the most natural generalizations of the problems considered in the
present paper is the following one. Let $F_d$ ($d\ge 1)$ be the free group
on generators $x_1,\dots ,x_d$, let $G$ be an abstract group, and let
$G\ast F_d$ denote the free product. Then to every $w_\Sigma\in G\ast F_d$ one can
associate the word map with constants
\begin{equation} \label{mapconst}
\w_\Sigma\colon G^d\to G
\end{equation}
defined by evaluation, exactly as for genuine word maps.
For the resulting equations with constants of the form
%\begin{equation}
$$
w_1(x_1,\dots ,x_d)\sigma_1\cdots
w_r(x_1,\dots ,x_d)\sigma_rw_{r+1}(x_1,\dots ,x_d)=g
$$
%\end{equation}
one can pose the same questions as those discussed in the present paper
for word maps without constants.
In particular, one can ask about the surjectivity or dominance of map
\eqref{mapconst}, about the size and structure of its image, etc. These
topics are almost unexplored and, in our opinion, definitely deserve
thorough investigation. Being interesting in its own right, say, in view of
natural connections with classical group-theoretic problems such as Thompson's conjecture and computing covering numbers (see, e.g., \cite{Go1}), information on
the properties of equations with constants can be useful for treating genuine
word equations; relevant examples can be found in \cite{GKP1}, \cite{GKP2}, \cite{KT}. Here we only quote several results from these papers. Recall
that as mentioned in Introduction, we limit ourselves to considering equations
{\it in} groups but not {\it over} groups.

Following \cite{KT}, we denote by $\e\colon G\ast F_d\to F_d$ the augmentation map,
sending all elements of $G$ to 1. %and call $\e(w_\Sigma)$ the content of $w_\Sigma$.
If $\e(w_\Sigma)=1$, we say that $w_\Sigma$ is singular.

With this notation, we have the following facts:

\begin{itemize}
\item[(i)] If $G=\Un (n)$, $d=1$, and a word with constants $w_\Sigma$ is non-singular,
then the map $\w_\Sigma\colon G\to G$ is surjective \cite{GeRo}.
\item[(ii)] If $p$ is a prime number, $G=\SU (p)$, $d=2$, and $\e(w_\Sigma)$ does not belong to
$[F_2,F_2]^p [F_2,[F_2,F_2]]$ (the second step of the exponent-$p$ central series),
then the map $\w_\Sigma\colon G\to G$ is surjective.
\item[(iii)] If $G$ is (the group of points of) a simple linear algebraic group defined over an algebraically closed field, $w_\Sigma=w_1\sigma_1\cdots w_r\sigma_rw_{r+1}$ is a non-singular word with $w_2,\dots ,w_{r+1}\ne 1$ and
``general'' $\sigma_1,\dots ,\sigma_r$, then the map $\w_\Sigma\colon G\to G$
is dominant \cite{GKP2} (see there a precise definition of a general $r$-tuple).
\end{itemize}

Note that the methods used to prove these statements are entirely different:
(i) relies on a purely homotopic approach (the Hopf degree theorem), (ii) needs
much more advanced techniques from homological algebra, and (iii) is based on
algebraic-geometric arguments.

How far can one hope to go trying to generalize these surjectivity and dominance
results? There are some immediate limitations: say, there are simple algebraic
groups and words with constants such that the image of map \eqref{mapconst}
collapses to 1 (so-called group identities with constants, see, e.g., \cite{Go2}).
The word $w_{\Sigma}(x)=\sigma^{-1}x\sigma$ gives rise to an example of
a map \eqref{mapconst} whose image consists of a single conjugacy class of $G$.
So far, the most optimistic approach consists in parameterization of the image of
\eqref{mapconst} using the quotient map $\pi\colon G\to T/W$, where $T$ is
a maximal torus of $G$ and $W$ is the Weyl group (see \cite{SS}). Namely, one can show (see \cite{GKP3}) that if the composed map
$\pi\circ \w_\Sigma\colon G^d\rightarrow T/W$ is dominant,
then so is the word map with constants $\w^\prime_\Sigma\colon G^{d+1}\rightarrow G$ corresponding to $w^\prime_\Sigma = yw_\Sigma y^{-1}$. Thus in such a case
the map $\w_\Sigma$ is ``dominant up to conjugacy'', or, in other words,
almost all conjugacy classes of $G$ (except for some closed subset of $G$)  intersect  $\Imm \w$.
So our biggest hope is the
dichotomy which will arise if the following question (see \cite{GKP3}) is answered in the affirmative.

\begin{quest}
Is it true that
%$$\Imm \pi\circ \w(x_1, \dots, x_d, \sigma_1, \dots, \sigma_r) = \begin{cases}%\text{either just one point
% for every }\\
% \Sigma = (\sigma_1, \dots, \sigma_r) \in G^r,\\
%\text{or}\\
%\text{a dense subset in }\,\,\,T/W\,\,\,\text{for}\\ \text{ every }\, \Sigma =
%(\sigma_1, \dots, \sigma_r) \in U\\\text{ from some open set}\,\,\,U\in G^r.
%\end{cases}$$
$\Imm (\pi\circ \w(x_1, \dots, x_d, \sigma_1, \dots, \sigma_r))$ is either just one point for every $\Sigma = (\sigma_1, \dots, \sigma_r) \in G^r$,
or a dense subset in $T/W$ for every $\Sigma = (\sigma_1, \dots, \sigma_r) \in U$ from some non-empty open set $U\subset G^r$?
\end{quest}

%\bigskip

\noindent{\it Acknowledgements.} %The research of the first author
%was supported by the Ministry of Education and Science of Russian Federation  project 1.661.2016/1.4.
%The research of the
%second and third authors was supported by ISF grant
%1623/16 and the Emmy Noether Research Institute for Mathematics.
%The paper was written when the second author visited the MPIM (Bonn).
%The authors thank all these institutions.
Discussions and correspondence with T.~Bandman, J.~Blanc,
S.~Malev, L.~Polterovich, A.~Rapinchuk, and E.~Shustin are gratefully
appreciated.

\end{document}